\documentclass[12pt]{amsart}

\usepackage{amssymb}

\usepackage{enumitem}

\usepackage{graphicx}

\usepackage[dvipsnames]{xcolor}
\usepackage[sort, numbers]{natbib}
\usepackage[utf8]{inputenc}
\usepackage[english,french]{babel}
\usepackage{amsmath}
\usepackage{amsthm}
\usepackage{bbold}
\usepackage[backref=page, colorlinks=true, citecolor=BlueViolet, linkcolor=BrickRed]{hyperref}
\usepackage{bm}
\usepackage{mathtools}

\makeatletter
\@namedef{subjclassname@2020}{%
  \textup{2020} Mathematics Subject Classification}
\makeatother

\usepackage[T1]{fontenc}

\newtheorem{theo}{Théorème}[section]

\newtheorem{cor}[theo]{Corollaire}
\newtheorem{prop}[theo]{Proposition}
\newtheorem{lem}[theo]{Lemme}

\newtheorem{conj}[theo]{Conjecture}

\theoremstyle{definition}
\newtheorem{defn}[theo]{Définition}
\newtheorem{remark}[theo]{Remarque}

\numberwithin{equation}{section}

\newcommand\R{\mathbb{R}}
\newcommand\C{\mathbb{C}}

\newcommand\Z{\mathbb{Z}}
\newcommand\N{\mathbb{N}}

\newcommand\1{\mathbb{1}}

\newcommand{\tendvers}[2]{\underset{#1 \to #2}{\longrightarrow}}

\newcommand{\ssum}[1]{\sum_{\substack{#1}}}
\newcommand{\dd}{\,\mathrm{d}}
\newcommand{\ee}{\mathrm{e}}
\newcommand{\oo}[1]{\mathopen{]}{#1}{\mathclose{[}}}
\newcommand{\of}[1]{\mathopen{]}{#1}{\mathclose{]}}}
\newcommand{\fo}[1]{\mathopen{[}{#1}{\mathclose{[}}}

\begin{document}


\baselineskip=17pt


\title[Un crible minorant effectif pour les entiers friables]{Un crible minorant effectif pour les entiers friables}

\author[Adrien Mounier]{Adrien Mounier}
\address{Institut de Mathématiques de Marseille\\ Aix-Marseille Université\\
163 Av. de Luminy\\
13009 Marseille, France}
\email{adrien.mounier@univ-amu.fr}

\date{}

\begin{abstract}
Let $\mathcal{A}$ 
be a finite set of positive integers and $y\geq 1$. We give an effective lower bound of the cardinality of the set $\{n\in\mathcal{A};\,p|n\Longrightarrow p\leq y\}$ under the condition of a good knowledge of the level of distribution of the set $\mathcal{A}$. Some consequences are studied: an application to the friable values of irreducible polynomials or binary forms with integer coefficients, and an application to shifted friable numbers.
\end{abstract}



\maketitle

{ \hypersetup{hidelinks} \tableofcontents }

\section{Introduction}
Notons $P^+(n)$ le plus grand facteur premier d'un entier naturel $n\geq 1$ avec la convention~$P^+(1) := 1$. Pour $y\geq 1$, l'entier $n$ est dit ``$y$-friable'' (ou $y$-smooth) lorsque $P^+(n) \leq y$.
L'ensemble des nombres premiers est  noté $\mathcal{P}$ et la lettre~$p$ désigne toujours un de ses éléments, ainsi que les notations~$p_1,p_2,\ldots$  Soit~$\mathcal{A}$ un ensemble fini non vide d'entiers naturels non nuls. Le but de cet article est de trouver une minoration explicite du nombre d'entiers $y$-friables dans~$\mathcal{A}$:  $$\Psi_\mathcal{A}(y) := \#\{n\in\mathcal{A}:P^+(n)\leq y\}.$$ 
Notons 
$$x := \sup \mathcal{A}\quad\text{et}\quad u:=\frac{\log x}{\log y}.$$
En particulier si~$\mathcal{A} = \N\cap [1,x]$, alors $\Psi_{\mathcal{A}}(y)$ devient la fonction sommatoire des entiers $y$-friables:
$$\Psi(x,y):=\#\{1\leq n\leq x:P^+(n)\leq y\}.$$
Cette fonction a fait l'objet d'une abondante littérature, à commencer par Dickman~\cite{Dic30} qui a montré en 1930 que pour tout $u > 0$ fixé, 
\begin{equation}\label{eq:Dic1}
    \Psi(x,x^{\frac{1}{u}}) \sim x\varrho(u)
\end{equation}
où $\varrho$, appelée fonction de Dickman, est l'unique fonction définie sur $\R_+$, continue, dérivable sur $]1,+\infty[$, valant~$1$ sur $[0,1]$ et vérifiant l'équation différentielle aux différences
$$u\varrho'(u) + \varrho(u-1) = 0\quad\quad(u>1).$$
Puis en 1986, suite notamment aux travaux de de Bruijn \cite{dBr66}, Hildebrand \cite{Hil86} élargit le domaine de validité de \eqref{eq:Dic1} et donne un ordre de grandeur du reste. Plus précisément, il fournit
\begin{equation}\label{eq:Hil1}
    \Psi(x,y) = x\varrho(u)\left\{1+O\left(\frac{\log(u+1)}{\log y}\right)\right\},
\end{equation}
valable dans le domaine 
$$\exp((\log\log x)^{\frac{5}{3}+\varepsilon})\leq y\leq x.$$
Des estimations encore plus précises ont ensuite fait l'objet de plusieurs articles, notamment celui de Saias \cite{Sai89} puis Hildebrand et Tenenbaum \cite{HT93}. Pour plus de détails sur la fonction $\Psi$, nous recommandons au lecteur le chapitre III.5 de \cite{Ten15}.\\

Revenons dans le cas général d'un ensemble fini d'entiers naturels $\mathcal{A}$ quelconque et supposons désormais que $u$ est dans un compact de $\fo{1,+\infty}$.
Si $\mathcal{A}$ n'a pas de structure multiplicative particulière et n'est pas trop épars, une heuristique nous suggère que
\begin{equation}\label{eq:heuristiquegenerale}
    \Psi_{\mathcal{A}}(y) \sim \varrho(u)\#\mathcal{A}
\end{equation} 
Nous minorons ici $\Psi_{\mathcal{A}}(y)$ dans le cas où le cardinal 
$$A_d := \#\{n\in\mathcal{A}:d\mid n\}\quad\quad(d\geq 1)$$
peut être estimé avec une certaine uniformité en $d$. Soient $X,\kappa$ des réels positifs et~$\gamma$ une fonction arithmétique multiplicative vérifiant les conditions suivantes
\begin{equation}\label{cond1}
    \exists A\geq 0,\,\forall z\geq 2,\,\left|\sum\limits_{p\leq z} \frac{\log p}{p}\gamma(p) - \kappa\log z \right|\leq A,\tag{A$_\kappa$}
\end{equation}
\begin{equation}\label{cond2}
    \exists B\geq 0,\,\forall p\in\mathcal{P},\,\forall\alpha\in\N^*,\,\gamma(p^\alpha) \leq B.\tag{B}
\end{equation}
Notons, pour $d\geq 1$
    $$r_d := A_d -\frac{\gamma(d)}{d}X,$$
et supposons qu'il existe $\theta\in\oo{1-\frac{1}{u},1}$ tel que 
\begin{align}
    \exists\delta>0,\;\sum_{d\leq y^{u\theta}}|r_d| \ll \frac{X}{(\log y)^\delta}\label{cond3}\tag{C$_\theta$}.
\end{align}
En particulier, en ne gardant dans la somme que le terme $d = 1$, nous constatons que $X$ doit être choisi tel que $X \sim \#\mathcal{A}$.
Le résultat principal se résume dans le théorème suivant: 
\begin{theo}\label{theo:principal}
   Soient $u\geq 1$ et $\theta\in\oo{0,1}$ tels que $(1-\theta)u < 1$. Alors sous les conditions~\eqref{cond1},~\eqref{cond2} et~\eqref{cond3}, 
    \begin{align*}
        \Psi_{\mathcal{A}}(y) \geq & \{\varrho_{\theta,\kappa}(u) + o(1)\}\#\mathcal{A}
    \end{align*}
    lorsque $x\to\infty$, où le coefficient $\varrho_{\theta,\kappa}(u)>0$ est défini en~\eqref{eq:rhotheta}. La convergence est uniforme pour~$u$ borné. De plus, $\varrho_{\theta,1}(u)\to\varrho(u)$  (fonction de Dickman) lorsque~$\theta~\to~1$. 
\end{theo}
\begin{remark}\label{rem:conda1}
    Dans la pratique, il peut être utile de remplacer la condition~\eqref{cond3} par la condition plus faible 
    \begin{equation}
        \exists \delta>0,\,\forall a\in\Z^*,\,\forall k\geq 1,\quad\sup_{\substack{U\subset \mathcal{S}(u\theta)\\ U\text{ convexe}}}\Bigg|\ssum{(\frac{\log p_1}{\log y},\ldots,\frac{\log p_k}{\log y})\in U\\(a,p_j)=1}r_{p_1\cdots p_k}\Bigg|\ll_{k,a} \frac{X}{(\log y)^\delta}\label{cond4},\tag{D$_\theta$}
    \end{equation}
    où $\mathcal{S}(\sigma) := \{\mathbf{s}\in[0,1]^k;\,\sum_is_i\leq \sigma\}$. La condition \eqref{cond4} a l'avantage de supprimer les valeurs absolues de $r_d$ dans \eqref{cond3}, ce qui pourrait offrir la possibilité d'exploiter la structure multilinéaire de la somme sur $p_1,\ldots,p_k$. De plus, la somme porte sur des entiers ayant un nombre fixé de facteurs premiers comptés avec multiplicité, cela sera utile dans une de nos applications (le Théorème \ref{theo:HaretPas}).
    Le fait que la condition~\eqref{cond4} reste suffisante est détaillé dans la démonstration du Théorème~\ref{theo:principal}.
\end{remark}

\begin{remark}
Alors que les méthodes de cribles habituelles s'attachent à estimer des sommes portant sur des entiers sans petits facteurs premiers, le Théorème \ref{theo:principal} fournit une minoration pour des sommes portant sur des entiers sans grands facteurs premiers. Les conditions \eqref{cond1} et \eqref{cond2} se retrouvent dans les deux cas (voir par exemple l'équation $(3.17)$ et Lemma $2.2$ du livre de Halberstam et Richert \cite{HR74} ou par exemple les équations $(4.8)$, $(4.12)$, $(5.34)$, ou la Proposition $3.4$ du livre de Friedlander et Iwaniec \cite{FI10}).

\end{remark}

Bien qu'ayant une formulation explicite, les fonctions $\varrho_{\theta,\kappa}$ qui interviennent dans le théorème ci-dessus ne sont en général pas simples à calculer directement. Ceci est cependant possible pour certaines valeurs de $u$ et $\theta$ qui se rencontrent dans un grand nombre d'applications.

\begin{prop}[Corollaire~\ref{cor:rhotheta}]
 Soient $u\geq 1$ et $\theta\in\mathopen{]}0,1\mathopen{]}$ des réels tels que~$(1-\theta)u \leq 1$. Soit $\varepsilon>0$. La fonction $\varrho_{\theta,\kappa}$ définie en \eqref{eq:rhotheta} et intervenant au Théorème~\ref{theo:principal} vérifie:
        \begin{itemize}
            \item Si $0<\theta \leq \frac{1}{2}$,
        \begin{equation}\label{eq:theta12}
            \varrho_{\theta,\kappa}(u) = 
                \kappa(1-(1-\theta)u)u^{\kappa-1},\quad\quad\quad\left(1<u\leq \tfrac{1}{1-\theta}\right).
        \end{equation}
        \item Si $\frac{1}{2}<\theta \leq \frac{2}{3}$,
        \begin{equation}\label{eq:theta23}
            \varrho_{\theta,1}(u) = \begin{cases}
                1-(1-\theta)u&\quad\text{si } 1<u\leq \frac{1}{\theta},\\
                1-(1-\theta)u+\log\left(\frac{1}{u}+1-\theta\right)&\quad\text{si } \frac{1}{\theta}<u\leq 2,\\
                1-(1-\theta)u+\log\left(\frac{1}{u}+1-\theta\right) + (1-\theta)u\log\frac{u}{2}&\quad\text{si } 2<u\leq \frac{1}{1-\theta}.
            \end{cases}
        \end{equation}
        \item Si $\theta = 1-\varepsilon$,
        \begin{equation}\label{eq:theta1}
            \varrho_{\theta,1}(u) = \varrho(u) + O(\varepsilon)\quad\text{si } 0<u.
        \end{equation}
        \end{itemize}
\end{prop}

Nous énonçons maintenant quelques conséquences intéressantes du Théorème~\ref{theo:principal}.

\subsection{Valeurs de polynômes}
Une première application est de se placer dans le cas où $\mathcal{A}$ est un ensemble de valeurs polynomiales: si $F\in\Z[X]$, il est naturel de s'intéresser aux propriétés multiplicatives de la suite $(F(n))_{n\in\N}$. Nous étudions ici la quantité
$$\Psi_F(x,y):=\#\{n\leq x:F(n)\neq 0,\;P^+(|F(n)|)\:\leq y\}.$$
Supposons que le polynôme $F$ est tel que les propriétés multiplicatives respectives des ses facteurs irréductibles  sont indépendantes les unes des autres. L'heuristique \eqref{eq:heuristiquegenerale} nous suggère alors l'équivalent asymptotique suivant (au moins pour $u\geq 1$ borné):
$$\Psi_F(x,y) \sim x\varrho(ud_1)\cdots\varrho(ud_k)$$
où $F=F_1\cdots F_k$ est la décomposition en produit de facteurs irréductibles de $F$ et $d_i=\deg(F_i)$. Plus précisément, Martin (\cite{Mar00}, Theorem 1.1) montre que sous une certaine hypothèse (notée UH \emph{ibid}), version quantitative de l'hypothèse H de Schinzel et Sierpiński, nous avons l'estimation souhaitée dans un certain domaine en $u$ et avec un terme d'erreur dépendants des paramètres du polynôme. En ce qui concerne les résultats inconditionnels sur la quantité $\Psi_F(x,y)$, nous disposons du théorème suivant dû à Dartyge, Martin et Tenenbaum:
\begin{theo}[\cite{DMT01}, Theorem $1.1$]\label{theo:DMT}
    Soit $\varepsilon>0$. Notons $g$ le degré du plus grand facteur irréductible de~$F$ et $k$ le nombre de facteurs irréductibles de~$F
    $ de degré $g$, alors 
    \begin{equation}\label{eq:DMT}
        \Psi_F(x,y) \asymp x
    \end{equation}
    lorsque $x$ est suffisamment grand, pourvu que $y\geq x^{g+\varepsilon-1/k}$.
\end{theo}

La force de cet énoncé réside dans le fait que $F$ est arbitraire mais le prix à payer est que les constantes de minoration sont très petites. Grâce au Théorème~\ref{theo:principal}, nous fournissons des constantes explicites d'un ordre de grandeur satisfaisant (voir théorème \ref{theo:Fn}) lorsque $F$ est irréductible. Commençons par un cas particulier classique pour fixer les idées et prenons $F(X) = X^2+1$. Le théorème~\eqref{theo:DMT} énonce alors que pour tout~$\varepsilon>0$, $\Psi_F(x,y)\asymp x$ dès que~$y\geq x^{1+\varepsilon}$.
Nous précisons la minoration dans le résultat suivant: 
\begin{theo}\label{theo:ncarre}
    Soit $F(X) = X^2+1$. Si $x < y \leq x^2$, alors 
    $$\Psi_F(x,y) \geq \{1-u+o(1)\}x$$
    lorsque $x\to\infty$. (Notons que $u:=\frac{\log x}{\log y}\in\fo{\frac{1}{2},1}$).
\end{theo}
Notons que dans ce cas particulier, Harman \cite{Har24} a montré que la minoration $\Psi_F(x,y)\gg x$ est valable pour $u< \frac{381}{356}\ee^{1/2}\approx 1.76$, ce qui est une amélioration significative des travaux précédents, qui commencèrent avec celui de Dartyge: $u<\frac{179}{149}\approx 1.20$ \cite{Dar96}, puis Harman:~$u<1.25$~\cite{Har08}. Les systèmes de poids utilisés dans ces travaux permettent d'élargir le domaine en $u$ au prix de réduire la constante implicite de minoration.

\begin{remark}
    Le coefficient de minoration $1-u$ du Théorème \ref{theo:ncarre} est issu de l'exposant de répartition $1/2-\varepsilon$ de la suite $(n^2+1)_{n\geq 1}$ dans les suite arithmétiques. Grâce aux travaux de Hooley \cite{Hoo67}, puis Deshouillers et Iwaniec \cite{DI82}, la Bretèche et Drappeau \cite{DdlB19}, Merikoski \cite{Mer23} et Pascadi \cite{Pas25}, il est sûrement possible d'obtenir un meilleur exposant de répartition et donc d'améliorer le coefficient de minoration. 
\end{remark}

Le théorème \ref{theo:ncarre} se généralise dans le résultat suivant.
\begin{theo}\label{theo:Fn}
    Soit $F\in Z[X]$ irréductible, de degré $g\geq 2$. La minoration
    $$\Psi_F(x,y) \geq \{1+u-gu+o(1)\}x $$
    est valable lorsque $x$ est suffisamment grand, pourvu que $y>x^{g-1}$. (Notons que $u:=\frac{\log x}{\log y}\in\fo{\frac{1}{g},\frac{1}{g-1}}$).
\end{theo}

\begin{remark}
    Le théorème \ref{theo:Fn} ci-dessus ne donne une minoration effective que pour des polynômes irréductibles, alors que le théorème~\ref{theo:DMT} traite également le cas où le polynôme $F$ a plusieurs facteurs irréductibles. Il est néanmoins possible, à priori, d'adapter notre méthode pour étendre notre résultat à ce cas général et d'obtenir 
    $$\Psi_F(x,y) \geq \{\rho_k(u,g) + o(1)\}x$$
    valable pour $x$ suffisamment grand et dans un domaine~$y>x^{g-1/k+\varepsilon}$ où $k$ désigne le nombre de facteurs irréductible de degré $g$ de $F$. Cela fera l'objet d'un futur travail.
\end{remark}

\begin{remark}
    En ce qui concerne la majoration de $\Psi_F(x,y)$, un récent travail de la Bretèche et Drappeau \cite{DdlB19} fournit
$$\Psi_F(x,y) \ll x\varrho(u)^{c_F+o(1)}$$
lorsque $x\to\infty$ et $(\log x)^{\psi(x)}\leq y\leq x$ (où $\psi(x)\to\infty$) 
et $c_F\geq 1$ une constante dépendant uniquement des facteurs irréductibles de $F$ de petits degrés. Il est à noter que $c_F=1$ lorsque tous les facteurs irréductibles de $F$ ont des degrés supérieurs ou égaux à $3$. Ce cas particulier a été initialement découvert par Khmyrova \cite{Khm64}.
\end{remark}

\subsection{Valeurs de formes binaires}
Une deuxième application consiste à considérer les valeurs friables prises par des formes binaires à coefficients entiers. Ce sujet a été étudié en 2012 par Balog, Blomer, Dartyge, et Tenenbaum dans \cite{BBDT12}.\\
Pour $F\in\Z[X,Y]$, notons encore 
$$\Psi_F(x,y) := \#\{1\leq a,b\leq x:F(a,b)\neq 0,\;P^+(|F(a,b)|)\leq y\}.$$
Les auteurs montrent le résultat suivant:

\begin{theo}[\cite{BBDT12}, Theorem 1]\label{theo:BBDT}
Soit $\varepsilon>0$ et $F(X,Y)$ une forme binaire à coefficients entiers, de degré $t\geq 2$ sans facteurs carrés. Notons $g$ le plus grand degré des facteurs irréductibles de $F$ et~$k$ (resp. $\ell$) le nombre de facteurs irréductibles de degré $g$ (resp. $g-1$). L'estimation
$$\Psi_F(x,y) \asymp_{F,\varepsilon}x^2$$
est valable pour $x$ suffisamment grand pourvu que $y \geq x^{\alpha_F+\varepsilon}$ où 
$$\alpha_F := \begin{cases}
    g-2/k,&\quad\text{si }k\geq 2,\\
    g-1-1/(\ell+1),&\quad\text{si }k=1\text{ et }(g,t)\neq(2,3),\\
    2/3,&\quad\text{si }(g,k,t)=(2,1,3).\\
\end{cases}$$
\end{theo}

En particulier, si la forme binaire $F$ est irréductible, alors $t=g$ et~$(k,\ell)=(1,0)$, donc~$\alpha_F = t-2$.
Dans ce cas, nous pouvons préciser le théorème ci-dessus grâce au Théorème~\ref{theo:principal}:

\begin{theo}\label{theo:Formebinaire}
    Soit $F(X,Y)$ une forme binaire à coefficients entiers, de degré~$t\geq 4$, irréductible. Avec les notations du théorème précédent, la minoration
    \begin{align*}
        \Psi_F(x,y) &\geq \{1-(t-2)u+o(1)\}x^2.\\
    \end{align*}
    est valable lorsque $x\to\infty$ et dans le domaine~$x^{t-2}\leq y$, \emph{i.e}~$u(t-2)\leq 1$.
\end{theo}
\begin{remark}\label{rem:Lachand}
    Dans le cas où $y\gg x^t$, l'estimation devient triviale. Par ailleurs, les cas des formes binaires irréductibles de degré $2$ et $3$ ont été étudiés par Lachand dans sa thèse \cite{Lac14}, puis dans~ \cite{Lac15,Lac17,Lac18}. En particulier dans sa thèse, il établit un équivalent asymptotique de~$\Psi_F(x,y)$ et dans des domaines beaucoup plus larges que le nôtre:\\
    
    Soit $\varepsilon>0$.
    \begin{itemize}
        \item Si $t=2$, alors  
    $$\Psi_F(x,y) = \varrho(2u)x^2 + O_\varepsilon\left(\frac{x^2\log\log x}{\log y}\right)$$
    uniformément dans le domaine $x\geq 2$, $\exp((\log x)^{2/5+\varepsilon})\leq y\leq x$
    (voir \cite{Lac14}, Chapitre.2, Théorème 4.2). \\
    \item Si $t=3$, alors  
    $$\Psi_F(x,y) = \varrho(3u)x^2 + O_\varepsilon\left(\frac{x^2}{(\log\log x)^{1-\varepsilon}}\right)$$
    uniformément en $x,y\geq 3$ (voir \cite{Lac14}, Chapitre.3, Corollaire 1.3).
    \end{itemize}
\end{remark}

\begin{remark}
    Le Théorème \ref{theo:Formebinaire} traite la friabilité des formes binaires irréductibles. Plus généralement, il est naturel de s'intéresser à la friabilité des polynômes à plusieurs variables. Même si ce cas n'est pas traité ici, le lecteur intéressé pourra se référer à l'article \cite{Fou10} de Fouvry.
\end{remark}

\subsection{Entiers friables voisins}
Une autre application du Théorème~\ref{theo:principal} est de compter les entiers friables voisins. Pour $u,v\geq 1$ et $a$ un entier naturel non nul, nous étudions la quantité 
$$\Psi(x;u,v) := \{1\leq n\leq x : P^+(n)\leq x^{1/u}\mbox{ et } P^+(n+a)\leq x^{1/v}\}.$$
L'heuristique~\eqref{eq:heuristiquegenerale} inspire la conjecture suivante:  
\begin{conj}\label{conj:rhourhov}
    Pour tous $u,v\geq 1$, 
    $$\Psi(x;u,v)\sim \varrho(u)\varrho(v)x$$
    lorsque $x\to\infty$.
\end{conj}
Cette conjecture avait initialement été formulée par Erd\H{o}s et Pomerance en 1978~\cite{EP78}.
Le Théorème~\ref{theo:principal} permet une minoration explicite. La condition~\eqref{cond4} découle d'estimations de type Bombieri-Vinogradov pour les entiers friables. Plus précisément, un théorème de Harper (\cite{Har12a}, \textit{Theorem $1$}) et un plus récent de Pascadi (\cite{Pas24}, \textit{Theorem $1.1$}) nous permettent d'énoncer le résultat suivant:

\begin{theo}\label{theo:HaretPas}
    Avec les mêmes notations que précédemment, la minoration
    $$\Psi(x;u,v)\geq \left\{f(u)+o(1)\right\}\Psi(x,x^{\frac{1}{v}})$$
    est valable lorsque $x\to\infty$, $v = o\left(\tfrac{\log x}{\log\log x}\right)$ et sous les conditions suivantes:
    \begin{itemize}
        \item Si $1<u\leq 2$ et
        \begin{equation}\label{eq:cas1HaretPas}
        f(u) = 1-\tfrac{u}{2}.
        \end{equation}
    \item Si $1<u\leq \tfrac{8}{3}$, $v\tendvers{x}{+\infty}\infty$ et \begin{equation}\label{eq:cas3HaretPas}
        f(u) =\begin{cases}
                1-\frac{3}{8}u,&\quad\text{si } 1<u\leq \frac{8}{5}\\
                1-\frac{3}{8}u+\log\left(\frac{1}{u}+\frac{3}{8}\right),&\quad\text{si } \frac{8}{5}<u\leq 2\\
                1-\frac{3}{8}u+\log\left(\frac{1}{u}+\frac{3}{8}\right) + \frac{3}{8}u\log\frac{u}{2},&\quad\text{si } 2<u\leq \frac{8}{3}
            \end{cases}
    \end{equation}
    \end{itemize}
\end{theo}

\begin{center}
\includegraphics[width=0.7\textwidth]{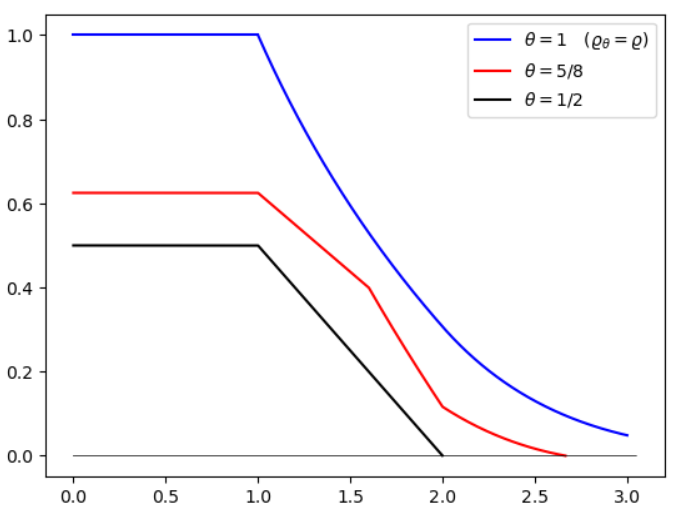}  
\end{center}

     Ci-dessus: une représentation graphique des fonctions $f$ correspondant aux deux cas ci-dessus (en noir et rouge). Les valeurs de $\theta$ indiquées en légende sont les exposants de répartition des entiers friables correspondants (voir partie~\ref{sec:voisinincond}). Le cas particulier $\theta = 1$ correspond à la valeur optimale conjecturée pour $f$, à savoir la fonction de Dickman $\varrho$. 

\begin{remark}
Dans \cite{Hil85}, Hildebrand prouve que 
$$\Psi(x;u,v)\gg x$$ lorsque~$u$ et $v$ sont bornés. Le Théorème~\ref{theo:HaretPas} est plus restrictif sur $u$ mais il a l’avantage de ne pas imposer
à $v$ d’être borné. De plus, la constante de minoration est effective.

Notons que sous la Conjecture~\ref{conj:EH} d'Elliot-Halberstam pour les entiers friables, la condition~\eqref{cond3} est vérifiée pour tout $\theta < 1$. Wang \cite{Wan21} a démontré que dans ce cas, la Conjecture~\ref{conj:rhourhov} d'Erd\H{o}s \& Pomerance est vraie. Par notre méthode, nous démontrons le corollaire suivant qui est la version minorante du résultat de Wang.
\end{remark}

\begin{cor}\label{cor:Wan}
    Soient deux réels $u\geq 1$, $v\geq 1$ et $a\in\Z^*$. Sous la Conjecture~\ref{conj:EH} d'Elliot-Halberstam pour les entiers friables, la minoration
    $$\Psi(x;u,v)\geq \{\varrho(u)\varrho(v)+o(1)\}x$$
    est valable lorsque $x\to\infty$.
\end{cor}

\begin{remark}
En 2022, Jiang, Lü et Wang \cite{JLW22} ont montré que la Conjecture~\ref{conj:rhourhov} est vraie pour tout $\varepsilon>0$ et presque tout $a\in\Z^*$ tel que~$|a|\leq H$ où $(\log x)^{1+\varepsilon}\leq H\leq x^{1-\varepsilon}$. 
\end{remark}

\section{Survol}

La démonstration du Théorème~\ref{theo:principal} part de l'inégalité suivante
\begin{equation}\label{eq:eqdebase}
\Psi_\mathcal{A}(y)\log x \geq \ssum{n\in\mathcal{A}\\P^+(n)\leq y}\log n\geq \sum_{p\leq y}\log p\ssum{n~:~pn\in\mathcal{A}\\P^+(n)\leq y}1
\end{equation}
où nous rappelons que $x:=\sup \mathcal{A}$. La dernière inégalité vient de l'identité de convolution $\log = 1*\Lambda$ (convolution de Dirichlet).

Cette relation est en fait une version simplifiée (exploitant uniquement la minoration) de l'identité de Hildebrand (voir (2) dans \cite{Hil86}) qui portait originellement sur l'ensemble $\mathcal{A} = \{1\leq~n\leq~x\}$. Il est à noter que l'identité de Hildebrand a également été utilisée par Granville dans \cite{Gra93b} avec~$\mathcal{A} = \{1\leq n\leq x:n\equiv a\pmod q\}$ afin d'étudier la répartition des entiers friables dans les progressions arithmétiques.\\

Pour $p\geq x/y$, les indices $n$ de la somme interne sont automatiquement $y$-friables (car inférieurs à $y$). Ainsi, 
\begin{align}\label{eq:Aplogp}
    \Psi_\mathcal{A}(y)\log x &= \ssum{p\leq y\\p\leq x/y}\log p\ssum{n~:~pn\in\mathcal{A}\\P^+(n)\leq y}1+\sum_{x/y< p\leq y}A_p\log p=:S_1 + S_2
\end{align}
À ce stade, nous minorons par $S_2$ ce qui fournit déjà une minoration effective de $\Psi_\mathcal{A}(y)$ à condition de savoir estimer $A_p$. En effet, sous les hypothèses~\eqref{cond1},~\eqref{cond2} et~\eqref{cond3}, et dans le cas optimiste où~$\theta \approx 1$, cela fournit la minoration effective
\begin{equation}\label{eq:minorefaible}
    \Psi_\mathcal{A}(y)\geq \kappa \left\{\frac{2}{u}-1+o(1)\right\}\#\mathcal{A}
\end{equation}
qui est non triviale pour $1\leq u< 2$.\\

Cependant, il est clair qu'une minoration aussi grossière ne peut donner un résultat optimal, il faut donc trouver un moyen de minorer $S_1$. Pour ce faire, nous utilisons la même minoration que dans la première inégalité de~\eqref{eq:eqdebase}, à savoir que $\log n\leq \log x$ pour $n\in\mathcal{A}$. De même,~$\log n\leq \log (x/p)$ pour $pn\in\mathcal{A}$, ce qui entraîne 
\begin{equation}\label{eq:eqdebase2}
    S_1 \geq \ssum{p\leq y\\p\leq x/y}\frac{\log p}{\log(x/p)}\ssum{n~:~pn\in\mathcal{A}\\P^+(n)\leq y}\log n.
\end{equation}
Nous pouvons alors à nouveau appliquer l'identité $\log = 1*\Lambda$, ce qui fournit après calculs
\begin{equation}\label{eq:eqdebase3}
    S_1 \geq \ssum{p_1,p_2<y\\p_1\leq x/y}\frac{\log p_1\log p_2}{\log(x/p_1)}\ssum{n~:~p_1p_2n\in\mathcal{A}\\P^+(n)\leq y}1.
\end{equation}
De même que précédemment, les indices $n$ de la somme interne sont $y$-friables dès que $p_1p_2> x/y$, ce qui nous amène à séparer la somme en deux:
\begin{equation}\label{eq:eqdebase4}
    S_1 \geq \ssum{p_1,p_2\leq y\\p_1p_2<x/y}\frac{\log p_1\log p_2}{\log(x/p_1)}\ssum{n~:~p_1p_2n\in\mathcal{A}\\P^+(n)\leq y}1 + \ssum{p_1,p_2\leq y\\p_1\leq x/y<p_1p_2<x}\frac{\log p_1\log p_2}{\log(x/p_1)}A_{p_1p_2} =: S_1' + S_2'.
\end{equation}
L'hypothèse~\eqref{cond3} et le premier théorème de Mertens permettent d'estimer la somme $S_2'$ . Ainsi, la minoration $S_1'\geq 0$ procure un terme principal meilleur que~\eqref{eq:minorefaible} et non trivial dans un domaine plus large, à savoir $1\leq u< 3$.

Dans $S_1'$, nous minorons le ``$1$'' par~$\log n/\log(x/p_1p_2)$ et poursuivons la récurrence en prenant soin à chaque étape d'isoler les grands nombres premiers comme nous l'avons fait ci-dessus en séparant $S_1$ et $S_2$, puis $S_1'$ et $S_2'$. Après $k$ itérations de la récurrence, nous obtenons une somme sur $p_1,\ldots,p_k$ dans l'intervalle~$\oo{x/y,x^\theta}$ (voir Proposition~\ref{prop:TT1T2R}). La borne ``$x^\theta$'' à droite est imposée par la condition~\eqref{cond3}. La borne ``$x/y$'' de gauche est là pour assurer que l'indice ``$n$'' dans~\eqref{eq:eqdebase} et~\eqref{eq:eqdebase3} est $y$-friable. Nous remarquons qu'il faut imposer la condition $x/y < x^\theta$ afin que la dernière somme ne soit pas vide, cette condition se réécrit $(1-\theta )u <1$ qui est bien l'hypothèse du théorème~\ref{theo:principal}. Cette somme est ensuite estimée grâce à une généralisation du premier théorème de Mertens qui fait l'objet du Lemme~\ref{lem:general}. Le terme principal dépend également de $\theta$ et lorsque celui-ci peut être choisi proche de $1$, le terme principal devient effectivement $\varrho(u)$, conformément à l'heuristique~\eqref{eq:heuristiquegenerale}.\\

Une autre manière d'exploiter la méthode ci-dessus consiste à appliquer la minoration de Hildebrand à tous les termes ``$n$'' dans la somme de droite de~\eqref{eq:eqdebase}, tandis que dans la méthode ci-dessus, nous l'avons appliquée uniquement à $S_1$. Cependant, cette manière de faire mène, en fin de compte, à un terme principal nul, sauf dans le cas $\theta=1$ (voir \cite{Mou23}) ou nous retrouvons~$\varrho(u)$.\\

Un autre amélioration de la méthode exposée ci-dessus pourrait être envisagée en repartant de~\eqref{eq:eqdebase} et en remplaçant la minoration $\log x \geq \sum_{d|n}\Lambda(d)$ par $(\log x)^k \geq \sum_{d|n}\Lambda_k(d)$ où~$\Lambda_k$ désigne la fonction de Von Mangoldt généralisée d'ordre $k$. Nous espérons ainsi récupérer de l'information car le support de $\Lambda_k$ croit avec $k$. Cependant, une première tentative avec $k=2$ et $\mathcal{A} = \{n\leq x;\,P^+(n+a)\leq z\}$ a révélé un terme principal moins bon que pour $k=1$. Cette piste n'a donc pas été explorée plus avant.

\section{Crible minorant pour les entiers friables}\label{sec:theoprincipal}

Dans cette partie, nous démontrons le Théorème~\ref{theo:principal}.
\subsection{Étude de $(\mu_k)_{k\geq 1}$}\label{sec:muk}
Dans cette section, nous explicitons certaines fonctions intervenant dans le calcul du coefficient $\varrho_\theta(u)$ du Théorème~\ref{theo:principal}.

\begin{defn}\label{def:Vketmuk}
    Pour $k\geq 2$ et $0\leq v\leq\min(1,u)$,
$$V_k(u,v) := \{\bm{t}\in\oo{0,1}^k:t_1 + \cdots + t_{k-1} < u-1    <t_1 + \cdots + t_k < u-v\}$$
et
\begin{equation}\label{eq:defmuk}
    \mu_k(u,v):= \frac{1}{u}\int_{V_k(u,v)}f_{k-1}(u;t_1,\ldots,t_{k-1})\dd \bm{t}
\end{equation}
où 
\begin{equation}\label{eq:fk}
    f_k(u;t_1,\ldots,t_k) = \prod\limits_{i=1}^k\frac{1}{u-t_1-\cdots-t_i}.
\end{equation}
De plus, nous posons par convention 
$$\mu_1(u,v) := \frac{1}{u}\mathrm{mes}(\oo{0,1}\cap\oo{u-1,u-v}).$$
\end{defn}

\begin{prop}\label{prop:recSk}
Soient $u$ et $v$ vérifiant $0\leq v< \min(1,u)$.\\
Pour tout $k\geq 2$, 
\begin{equation}\label{eq:recmuk}
    u\mu_k(u,v) =\begin{cases}
    0&\quad\text{si } 0<u\leq 1,\\
    \int_{\max(1,u-1)}^u \mu_{k-1}\left(s,v\right)\dd s&\quad\text{si }u>1
    \end{cases}       
\end{equation}
et  pour $k=1$,
\begin{equation}\label{eq:mu1}
    u\mu_1(u,v) = \begin{cases}
        u-v&\quad\text{si }0<u\leq 1,\\
        \min(1-v,2-u)&\quad\text{si }1<u\leq 2,\\
        0&\quad\text{si } u>2.
    \end{cases}
\end{equation}
\end{prop}

\begin{proof}
Le point~\eqref{eq:mu1} est une simple disjonction de cas en partant de la définition de~$\mu_1(u,v)$.\\
Pour $k\geq2$, l'idée est d'effectuer le changement de variables qui envoie $t_1$ sur $u-s$ et qui laisse~$t_2,\ldots,t_k$ inchangées:
\begin{align*}
    u\mu_k(u,v) &= \int_0^{\min(1,u-1)}\frac{\dd t_1}{u-t_1}\int_{V_{k-1}(u-t_1,v)}\prod\limits_{i=2}^k\frac{\dd t_i}{u-t_1-\cdots-t_i}\\
    &= \int_{\max(1,u-1)}^u\frac{\dd s}{s}\int_{V_{k-1}(s,v)}\prod\limits_{i=2}^k\frac{\dd t_i}{s-t_2-\cdots-t_i}\\
    &= \int_{\max(1,u-1)}^u\mu_{k-1}(s,v)\dd s.
\end{align*}
\end{proof}

\begin{prop}
    Pour $0\leq v<\min(1,u)$ et $\kappa>0$, la série $\sum_{k\geq 1}\kappa^k\mu_k(u,v)$ est convergente. Plus précisément, 
    $$\sum_{k\geq 1}\kappa^k\mu_k(u,v) \leq\begin{cases} 
    \kappa\quad\text{si }0<u\leq 1\\
    (1+u)^\kappa\quad\text{si }u>1.
    \end{cases}$$
\end{prop}

\begin{proof}
    Pour $k\geq 1$, notons $F_k(u,v) := \sum_{i=1}^k\kappa^i\mu_i(u,v)$. Comme
    chaque $\mu_i(u,v)$ est positif, la suite $(F_k(u,v))_{k\geq 1}$ est croissante. Montrons qu'elle est majorée. D'après~\eqref{eq:recmuk} et~\eqref{eq:mu1}, pour $k\geq 2$,
    \begin{equation}\label{eq:recFk}
        uF_k(u,v) = \kappa u\mu_1(u,v) + \kappa\int_{\max(1,u-1)}^uF_{k-1}(s,v)\dd s\quad\quad\quad (1\leq u)
    \end{equation}
    et 
    \begin{equation}\label{eq:Fkupetit}
        uF_k(u,v) = \kappa(u-v)\quad\quad\quad (0<u\leq 1).
    \end{equation}
    De~\eqref{eq:Fkupetit}, nous déduisons que $F_k(u,v) \leq \kappa $ pour $u\in\of{0,1}$. Supposons maintenant que $u>1$ et montrons que $F_k(u,v) \leq (1+u)^\kappa$ pour tout $k\geq 1$. 
    \begin{itemize}
        \item Pour $1<u\leq 2$, l'équation~\eqref{eq:mu1} implique~$u\mu_1(u,v) \leq 1$, donc l'équation~\eqref{eq:recFk} entraîne
    $$uF_k(u,v) \leq \kappa + \kappa\int_1^uF_{k-1}(s,v)\dd s$$
    et donc une récurrence directe fournit $F_k(u,v) \leq (1+u)^\kappa$ pour $k\geq 1$.
    \item  Pour $u>2$ une récurrence permet aussi de conclure car l'équation~\eqref{eq:recFk} devient
    $$uF_k(u,v) = \kappa\int_{u-1}^uF_{k-1}(s,v)\dd s.$$
    \end{itemize}
     
    Ainsi, la suite $(F_k(u,v))_{k\geq 1}$ converge. 
\end{proof}

Notons pour $u>0$ et $0\leq v<\min(1,u)$:
\begin{equation}\label{eq:mudef}
    \mu^{(\kappa)}(u,v) := \sum_{k=1}^\infty\kappa^k\mu_k(u,v)
\end{equation}

\begin{prop}\label{prop:mu}
    Pour $u>0$ et $0\leq v<\min(1,u)$,  
    
    \begin{itemize}
        \item Si $\kappa\neq 1$:
         \begin{equation}\label{eq:muuvkappa}
        \mu^{(\kappa)}(u,v) = \begin{cases}
            \kappa(1 - \frac{v}{u})&\quad\quad\quad\text{si }0<u\leq 1,\\
            \kappa(1 - v)u^{\kappa -1} &\quad\quad\quad\text{si }1<u\leq 1+v,\\
           \frac{v-v\kappa-1}{\kappa-1}\left(\frac{u}{1+v}\right)^{\kappa-1}+\frac{\kappa}{\kappa-1} &\quad\quad\quad\text{si }1+v<u\leq 2
        \end{cases}
    \end{equation}
    et
    \begin{equation}\label{eq:muuv2kappa}
        u\frac{\partial\mu^{(\kappa)}}{\partial u}(u,v)+(1-\kappa)\mu^{(\kappa)}(u,v)+\kappa\mu^{(\kappa)}(u-1,v) = 0\quad\text{si }2<u.
    \end{equation}
    \item Si $\kappa= 1$, $\mu^{(1)}(u,v)$ s'obtient en remplaçant $\kappa$ par $1$ dans les expressions de \eqref{eq:muuvkappa} et \eqref{eq:muuv2kappa} sauf dans le cas $1+v<u\leq 2$ où il faut prendre la limite lorsque $\kappa\to 1$, soit~$\mu^{(1)}(u,v) =  1 - v + \log\left(\frac{1+v}{u}\right)$. En particulier, $\mu^{(1)}(\cdot,0)$ est la fonction $\varrho$ de Dickman et plus généralement 
    \begin{equation}\label{eq:distrhomu}
        0\leq \varrho(u) - \mu^{(1)}(u,v) \leq \frac{v}{\Gamma(u+1)}\quad\quad\quad (u\geq 1).
    \end{equation}
    \end{itemize}

\end{prop}

\begin{proof}
    En faisant tendre $k$ vers $+\infty$ dans~\eqref{eq:recFk} et par le théorème de convergence dominée, il vient
    $$u\mu^{(\kappa)}(u,v) = \begin{cases}
        \kappa(u-v),&\quad\text{si } 0<u\leq 1,\\
        \kappa(1-v) + \kappa\int_{1}^u\mu^{(\kappa)}(s,v)\dd s&\quad\text{si } 1<u\leq 1+v,\\
        \kappa(2-u) + \kappa\int_{1}^u\mu^{(\kappa)}(s,v)\dd s&\quad\text{si } 1+v<u\leq 2,\\
        \kappa\int_{u-1}^u\mu^{(\kappa)}(s,v)\dd s&\quad\text{si } 2<u. 
    \end{cases}$$
    En particulier, nous remarquons que $\mu^{(\kappa)}(\cdot,v)$ est continue sur $\oo{0,+\infty}$ et dérivable sur chacun des intervalles ci-dessus, sauf éventuellement sur leurs bords, en lesquels une dérivée à gauche et à droite existent mais ne sont pas forcément égales. En dérivant sur chacun des $3$ derniers intervalles par rapport à $u$, nous obtenons
    $$u\frac{\partial\mu^{(\kappa)}}{\partial u}(u,v)+\mu^{(\kappa)}(u,v) = \begin{cases}
        \kappa\mu^{(\kappa)}(u,v)&\quad\text{si } 1<u\leq 1+v,\\
        -\kappa + \kappa\mu^{(\kappa)}(u,v)&\quad\text{si } 1+v<u\leq 2,\\
        \kappa\mu^{(\kappa)}(u,v)-\kappa\mu^{(\kappa)}(u-1,v)&\quad\text{si } 2<u. \\
    \end{cases}$$
    
    Enfin, \eqref{eq:muuvkappa} et \eqref{eq:muuv2kappa} se retrouvent en résolvant chaque équation différentielle, les conditions aux bords étant données par la continuité de $\mu^{(\kappa)}(\cdot,v)$.\\
    
    Montrons le point~\eqref{eq:distrhomu}. La minoration découle de la décroissance de $\mu^{(1)}(u,\cdot)$ et du fait que~ $\mu^{(1)}(\cdot,0)=\varrho$. Pour la majoration, nous procédons par récurrence sur $k:=\lfloor u\rfloor$:
    \begin{itemize}
        \item Si $k=1$, il faut distinguer les cas $1\leq u< 1+v$ et $1+v\leq u<2$. Dans le premier cas,~$\varrho(u) - \mu^{(1)}(u,v) = v-\log u < \frac{v}{u}$ car $v< 1$ et dans le deuxième cas, $\varrho(u) - \mu^{(1)}(u,v) = v-\log (1+v) \leq \frac{v^2}{2} < \frac{v}{u}$ car $v \leq u-1< \frac{2}{u}$. Le résultat découle ensuite du fait que $\Gamma(u+1) \leq u$ pour $1\leq u<2$.\\
        \item Soit $k\geq 2$ tel que~\eqref{eq:distrhomu} est montrée au rang $k-1$. Soit $u$ tel que~$k:=\lfloor u\rfloor$. Comme $k\geq 2$, nous avons vu plus haut que 
    $$u\mu^{(1)}(u,v) = \int_{u-1}^u\mu^{(1)}(s,v)\dd s$$
    et donc 
    \begin{equation}\label{eq:rhomoinsmu}
        \varrho(u)-\mu^{(1)}(u,v) = \frac{1}{u}\int_{u-1}^u(\varrho(s)-\mu^{(1)}(s,v))\dd s.
    \end{equation}
    Cette relation nous permet de montrer que $\varrho - \mu^{(1)}(\cdot,v)$ est décroissante. En effet, en dérivant par rapport à $u$, nous obtenons 
    $$\frac{\mathrm{d}}{\mathrm{d}u}(\varrho-\mu^{(1)}(.,v))(u) = -\frac{\varrho(u-1)-\mu^{(1)}(u-1,v)}{u} $$
    et $\varrho(u-1) \geq \mu^{(1)}(u-1,v)$ d'après la minoration dans \eqref{eq:distrhomu} montrée plus haut. Finalement, en majorant l'intégrale dans \eqref{eq:rhomoinsmu} par la valeur maximale de l'intégrande, nous obtenons
    $$\varrho(u)-\mu^{(1)}(u,v) \leq \frac{\varrho(u-1)-\mu^{(1)}(u-1,v)}{u}\leq \frac{v}{u\Gamma(u)} = \frac{v}{\Gamma(u+1)},$$
    la deuxième inégalité étant donnée par l'hypothèse de récurrence.
    \end{itemize}
    \end{proof}

    L'équation différentielle aux différences \eqref{eq:muuv2kappa} ressemble à celle vérifiée par la fonction $\varrho$ de Dickman, ce qui suggère qu'une expression explicite de ces fonctions n'existe pas dans le cas général. Cependant, cela s'avère possible pour certaines valeurs de $u$ et $v$ qui sont utiles dans les démonstrations des théorèmes énoncés dans l'introduction.\\
    
    Pour $\theta\in \of{1-1/u,1}$, notons
        \begin{equation}\label{eq:rhotheta}
            \varrho_{\theta,\kappa}(u) := \mu^{(\kappa)}(u,(1-\theta)u).
        \end{equation}
    Le coefficient $\varrho_{\theta,\kappa}(u)$ intervient dans le Théorème~\ref{theo:principal}. Le corollaire suivant permet de l'expliciter dans certains domaines en $u$ et $\theta$ qui se rencontrent dans les applications.
    \begin{cor}\label{cor:rhotheta}
    Soient $\theta,\varepsilon\in\of{0,1}$.
        \begin{itemize}
            \item Si $0<\theta \leq \frac{1}{2}$,
        \begin{equation}\label{eq:theta12}
            \varrho_{\theta,\kappa}(u) = 
                \kappa(1-(1-\theta)u)u^{\kappa-1}.\quad\quad\quad\left(1<u\leq \tfrac{1}{1-\theta}\right)
        \end{equation}
        \item Si $\frac{1}{2}<\theta \leq \frac{2}{3}$,
        \begin{equation}\label{eq:theta23}
            \varrho_{\theta,1}(u) = \begin{cases}
                1-(1-\theta)u,&\quad\text{si } 1<u\leq \frac{1}{\theta}\\
                1-(1-\theta)u+\log\left(\frac{1}{u}+1-\theta\right),&\quad\text{si } \frac{1}{\theta}<u\leq 2\\
                1-(1-\theta)u+\log\left(\frac{1}{u}+1-\theta\right) + (1-\theta)u\log\frac{u}{2}.&\quad\text{si } 2<u\leq \frac{1}{1-\theta}
            \end{cases}
        \end{equation}
        \item Si $\theta = 1-\varepsilon$,
        \begin{equation}\label{eq:theta1}
            \varrho_{\theta,1}(u) = \varrho(u) + O(\varepsilon)\quad\text{si } 0<u
        \end{equation}
        \end{itemize}
    \end{cor}
    \begin{proof}
        Les points~\eqref{eq:theta12} et~\eqref{eq:theta23} découlent directement de~\eqref{eq:muuvkappa} et le point~\eqref{eq:theta1} découle de~\eqref{eq:distrhomu}.
    \end{proof}

\subsection{Lemmes préliminaires}
Dans cette section, nous énonçons quelques lemmes indispensables à la démonstration du Théorème~\ref{theo:principal}.\\

Notons 
$$\Psi_{\mathcal{A}}(y,d) := \#\{n\in\mathcal{A}:P^+(n)\leq y,\,d|n\}.$$

Le lemme suivant est plutôt une remarque mais il nous a paru utile de le mettre en valeur car c'est l'ingrédient clé qui fait fonctionner notre méthode. En effet, il permet de supprimer la dépendance en $y$ du terme $\Psi_{\mathcal{A}}(y,d)$ pour des choix de $d$ suffisamment grand. C'est cette astuce qui permet de dégager la formule de récurrence qui fait l'objet de la Proposition~\ref{prop:TT1T2R} qui est centrale dans la démonstration du Théorème~\ref{theo:principal}.

\begin{lem}\label{lem:PsiAd}
    Si $d$ est un entier naturel vérifiant $y^{u-1}\leq d$ et $P^+(d)\leq y$, alors 
    $$\Psi_{\mathcal{A}}(y,d) = A_d.$$
\end{lem}

\begin{proof}
Il suffit d'écrire
    \begin{align*}
        \Psi_{\mathcal{A}}(y,d) &= \{n\in\mathcal{A}:P^+(n)\leq y,\, d|n\}\\
        &= \left\{ kd\in\mathcal{A}:P^+(kd)\leq y \right\}\\
        &= \left\{ kd\in\mathcal{A}:P^+(k)\leq y \right\}\quad\text{(car }P^+(d)\leq y).
    \end{align*}
    Or, la condition $y^{u-1} \leq d$ implique $\frac{x}{d} \leq y$, donc les $k$ apparaissant dans l'ensemble ci-dessus sont inférieurs à $y$ et \textit{a fortiori} $y$-friables, d'où le résultat.
\end{proof}
Pour tous nombres premiers $p_1,\ldots,p_k$ et tout réel $x> p_1\cdots p_k$, notons 
    $$Q(x;p_1,\ldots,p_k) := \frac{\log p_1\cdots\log p_k}{\log\left(\frac{x}{p_1}\right)\cdots\log\left(\frac{x}{p_1\dots p_k}\right)}.$$

Cette quantité apparaît naturellement dans la récurrence mise en place dans la Proposition~\ref{prop:TT1T2R}. Le lemme suivant est utile dans la démonstration du Lemme~\ref{lem:minQE}.

\begin{lem}\label{lem:minQ}
Soient $p_1,\ldots,p_k$ des nombres premiers.
Si $\theta\in\oo{0,1}$ est tel que $p_1\cdots p_k \leq x^{\theta}$, alors
    $$Q(x;p_1,\ldots,p_k) \ll_{\theta,k}1.$$
\end{lem}

\begin{proof}
    Pour tout $1\leq i\leq k$, $\log\left(\frac{x}{p_1\cdots p_i}\right)\geq (1-\theta)\log x$, d'où
    $$Q(x;p_1,\ldots,p_k) \leq \frac{1}{(1-\theta)^k(\log x)^k}(\log p_1)\cdots (\log p_k).$$
    Le résultat découle ensuite du fait que $\log p_i\leq \theta \log x$ pour chaque $i$.
\end{proof}

Le lemme suivant est utile dans la minoration de l'erreur qui apparaît dans la démonstration du Théorème~\ref{theo:principal} (voir section~\ref{sec:demotheoprincipal}).

\begin{lem}\label{lem:minQE}
    Soit $\theta\in\oo{1-\frac{1}{u},1}$.\\
    Sous la condition~\eqref{cond3}, il existe $\delta>0$ tel que pour tout $k\geq 1$,
    \begin{equation}\label{eq:sumEQ1}
        \sum_{\substack{p_1,\ldots,p_{k}\leq y\\ p_1\cdots p_{k}\leq y^{u\theta}}}|r_{p_1\cdots p_{k}}|Q(y^u;p_1,\ldots,p_{k-1}) \log p_{k}\ll_{k,\theta}X(\log y)^{1-\delta}.
    \end{equation}
    Sous la condition~\eqref{cond4}, il existe $\delta>0$ tel que pour tous $a\in\Z^*$ et $k\geq 1$,
    \begin{equation}\label{eq:sumEQ2}
        \sum_{\substack{p_1,\ldots,p_{k}\leq y\\ p_1\cdots p_{k}\leq y^{u\theta}\\p_1\cdots p_{k-1} \leq y^{u-1} \\(p_1\cdots p_k,a)=1}}r_{p_1\cdots p_{k}}Q(y^u;p_1,\ldots,p_{k-1}) \log p_{k}\ll_{a,k,\theta}X(\log y)^{1-\delta}.
    \end{equation}
\end{lem}

\begin{proof}
Notons $S$ le terme de gauche dans~\eqref{eq:sumEQ1}.
    \begin{align*}
        S\leq & \log y\sum_{\substack{p_1,\ldots,p_{k}\leq y\\ p_1\cdots p_{k}\leq y^{u\theta}}}|r_{p_1\cdots p_{k}}|Q(y^u;p_1,\ldots,p_{k-1}).
        \end{align*}
        Le Lemme~\ref{lem:minQ} permet ensuite de majorer $Q$:
        \begin{align*}
        S &\ll_{\theta,k} \log y\sum_{\substack{p_1,\ldots,p_{k}\leq y\\ p_1\cdots p_{k}\leq y^{u\theta}}}|r_{p_1\cdots p_{k}}|\\
        &\ll_{\theta,k} \log y\sum_{\substack{d\leq y^{u\theta}\\ \Omega(d) = k}}|r_{d}|\\
        &\ll_{\theta,k} \log y\sum_{d\leq y^{u\theta}}|r_{d}|\\
        &\ll_{\theta,k}X(\log y)^{1-\delta}\quad\text{d'après la condition }\eqref{cond3}.
    \end{align*}
   
   Notons $S'$ la somme du membre de gauche dans~\eqref{eq:sumEQ2}. 
   \begin{align*}
       S' &= \log y\sum_{\substack{p_1,\ldots,p_{k}\leq y\\ p_1\cdots p_{k}\leq y^{u\theta}\\p_1\cdots p_{k-1} \leq y^{u-1} \\(p_1\cdots p_k,a)=1}}r_{p_1\cdots p_{k}}F(\tfrac{\log p_1}{\log y},\ldots,\tfrac{\log p_{k-1}}{\log y})
   \end{align*}
   où $F(t_1,\ldots,t_k)=t_1\cdots t_kf_{k-1}(u;t_1,\ldots,t_{k-1})$
   et les fonctions $f_k$ sont définies en~\eqref{eq:fk}. Maintenant, pour $p_1,\ldots,p_k$ dans le domaine défini par les conditions de la somme ci-dessus, notons $t_i := \tfrac{\log p_i}{\log y}$. Nous pouvons écrire
   \begin{align*}
       F(t_1,\ldots,t_k) = \int_0^{t_1}\cdots\int_0^{t_k}\frac{\partial^k}{\partial s_1\cdots \partial s_k}F(s_1,\ldots,s_k)\dd \mathbf{s}.
   \end{align*}
   Donc
   \begin{align*}
       S' = \log y \underset{\substack{0\leq s_i\\\sum_i s_i \leq u\theta}}{\idotsint}\frac{\partial^k}{\partial s_1\cdots \partial s_k}F(s_1,\ldots,s_k)\sum_{\substack{\forall i,\,y^{s_i}\leq p_i\leq y\\ p_1\cdots p_{k}\leq y^{u\theta}\\p_1\cdots p_{k-1} \leq y^{u-1} \\(p_1\cdots p_k,a)=1}}r_{p_1\cdots p_{k}}\dd \mathbf{s}.
   \end{align*}
   D'après la condition~\eqref{cond4}, comme les conditions sur les $\log(p_i)/\log(y)$ dans la somme interne décrivent un convexe de $\R_+^n$, cette dernière est majorée par un $O_{k,\theta,a}(\frac{x}{(\log y)^\delta})$ uniformément en les~$s_i$. D'autre part, la fonction $F$ est $\mathcal{C}^\infty$ sur un ouvert qui contient l'ensemble $\mathcal{S}( u\theta)$, donc~$S' = O_{k,\theta,a}(x(\log y)^{1-\delta})$.

\end{proof}

Le lemme suivant est une généralisation à plusieurs dimensions du premier théorème de Mertens (voir par exemple \cite{Ten15}, Théorème. 1.8). Il permet de relier les sommes obtenues dans la Proposition~\ref{prop:TT1T2R} aux fonctions $\mu_k$ étudiées dans la partie~\ref{sec:muk}.

\begin{lem}\label{lem:general}
    Soient~$k\in\N$ et $C = C(k) > 0$. Soient~$U\subset [0, 1]^k$ un ensemble convexe, et~$f:U\to\C$ une fonction de classe~$\mathcal{C}^1$ vérifiant la condition 
    \begin{equation}\label{eq:condlemgen}
     \|f\|_\infty + \sum_{j=1}^k \Big\| \frac{\partial f}{\partial t_j}\Big\|_\infty \leq C. \tag{$\star$}
    \end{equation}
  Alors sous la condition~\eqref{cond1},
  $$ \ssum{p_1, \dotsc, p_k\in \mathcal{P} \\ \left(\frac{\log p_j}{\log y}\right)_{1\leq j\leq k}\in U} f\left(\frac{\log p_1}{\log y},\ldots,\frac{\log p_k}{\log y}\right)\prod\limits_j\frac{\gamma(p_j)\log p_j}{p_j\log y}  = \kappa^k\int_U f(\bm{t})\dd\bm{t} + O_k\left(\frac{C}{\log y}\right). $$
\end{lem}

\begin{proof}
  Nous procédons par récurrence. Pour $k=1$, l'ensemble $U$ est un sous-intervalle de~$[0, 1]$, notons $a$ et $b$ ses extrémités ($a<b$). Sous les conditions~\eqref{cond1} et~\eqref{eq:condlemgen}, une transformation d'Abel suffit:
  \begin{align*}
    \ssum{p\in \mathcal{P} \\ \tfrac{\log p}{\log y}\in U} \tfrac{\gamma(p)\log p}{p\log y} &f\left(\tfrac{\log p}{\log y}\right) = \ssum{y^a< p < y^b} \tfrac{\gamma(p)\log p}{p\log y} f\left(\tfrac{\log p}{\log y}\right)\\
    &= f(b) \ssum{y^a< p < y^b} \tfrac{\gamma(p)\log p}{p\log y} - \int_{y^a}^{y^b} \left(\ssum{y^a< p < v} \tfrac{\gamma(p)\log p}{p\log y}\right)f'\left(\tfrac{\log v}{\log y}\right)  \tfrac{\dd v}{v\log y} \\
    &= \kappa f(b)\left(b-a+O\left(\tfrac{1}{\log y}\right)\right) - \kappa\int_a^b \left(t-a+O\left(\tfrac{1}{\log y}\right)\right)f'(t)\dd t\\
    &= \kappa\int_a^bf(t)\dd t + O\left(\tfrac{C}{\log y}\right).
  \end{align*}
  Soit~$k\geq 2$ tel que le lemme est vérifié en $q$ variables pour tout $1\leq q\leq k-1$. Pour chaque~$p_1, \dots, p_{k-1}$, notons
  \begin{align*}
      I(p_1, \dotsc, p_{k-1}) &:= \left\{t\in\R_+: \left(\tfrac{\log p_1}{\log y},\ldots,\tfrac{\log p_{k-1}}{\log y},t\right)\in U\right\},\\
      f_{p_1, \dotsc, p_{k-1}}(t) &:= f\left(\tfrac{\log p_1}{\log y}, \dotsc, \tfrac{\log p_{k-1}}{\log y}, t\right).
  \end{align*}
  La fonction $f_{p_1, \dotsc, p_{k-1}}$ est de classe $\mathcal{C}^1$ et
  $$ \|f_{p_1, \dotsc, p_{k-1}}\|_\infty + \|f'_{p_1, \dotsc, p_{k-1}}\|_\infty \leq C. $$
  Puisque~$U$ est convexe, l'ensemble~$I(p_1, \dotsc, p_{k-1})$ décrit un intervalle, qui est inclus dans~$[0, 1]$. Nous pouvons donc appliquer l'hypothèse de récurrence au rang $1$ à la fonction $f_{p_1,\ldots,p_{k-1}}$: 
  $$ \ssum{p_k\in\mathcal{P} \\ \tfrac{\log p_k}{\log y}\in I(p_1,\ldots,p_{k-1})} \tfrac{\gamma(p_k)\log p_k}{p_k \log y} f_{p_1, \dotsc, p_{k-1}}\left(\tfrac{\log p_k}{\log y}\right) = \kappa\int_{I(p_1, \dotsc, p_{k-1})} f_{p_1, \dotsc, p_{k-1}}(\bm{t})\dd\bm{t} + O\left(\tfrac{C}{\log y}\right) $$
  uniformément vis-à-vis de~$p_1, \dotsc, p_{k-1}$. Soit $\Phi$ l'application définie sur $\R^k$ et à valeurs dans $\R^{k-1}$ définie par $\Phi(t_1,\ldots,t_k) := (t_1,\ldots,t_{k-1})$. 
  \begin{align*}
    S:=&  \ssum{p_1, \dotsc, p_k\in\mathcal{P} \\ \left(\tfrac{\log p_j}{\log y}\right)_{1\leq j\leq k}\in U} f\left(\tfrac{\log p_1}{\log y},\ldots,\tfrac{\log p_k}{\log y}\right)\prod\limits_j\tfrac{\gamma(p_j)\log p_j}{p_j\log y} \\
    ={}& \ssum{p_1, \dotsc, p_{k-1} \in\mathcal{P}\\ \left(\tfrac{\log p_j}{\log y}\right)_{1\leq j\leq k-1}\in \Phi(U)}  \left(\kappa\int_{t_k \in I(p_1, \dotsc, p_{k-1})} f_{p_1, \dotsc, p_{k-1}}(t_k) \dd t_k + O\left(\tfrac{C}{\log y}\right) \right)\prod_{j=1}^{k-1} \tfrac{\gamma(p_j)\log p_j}{p_j \log y} \\
    ={}& \kappa\ssum{p_1, \dotsc, p_{k-1} \in\mathcal{P}\\ \left(\tfrac{\log p_j}{\log y}\right)_{1\leq j\leq k-1}\in \Phi(U)}  \left( \int_{t_k \in I(p_1, \dotsc, p_{k-1})} f_{p_1, \dotsc, p_{k-1}}(t_k) \dd t_k\right)\prod_{j=1}^{k-1}\tfrac{\gamma(p_j)\log p_j}{p_j \log y}\\
    &\quad + \tfrac{C}{\log y}O\biggl(\ssum{p_1, \dotsc, p_{k-1} \in\mathcal{P}\\ \bigl(\tfrac{\log p_j}{\log y}\bigr)_{1\leq j\leq k-1}\in \Phi(U)}\prod_{j=1}^{k-1}\tfrac{\gamma(p_j)\log p_j}{p_j \log y}\biggr).
  \end{align*}
  D'une part, dans la somme principale, la condition $\left(\frac{\log p_j}{\log y}\right)_{1\leq j\leq k-1}\in \Phi(U)$ peut être oubliée car dès qu'elle n'est pas vérifiée, l'intégrale devient nulle en vertue de la définition de $I(p_1,\ldots,p_{k-1})$. D'autre part, $\Phi(U)\subset [0,1]^{k-1}$ donc la somme dans le $O$ peut être majoré par
    \begin{align*}
        \sum_{p_1, \dotsc, p_{k-1}\leq y}\prod_{j=1}^{k-1}\tfrac{\gamma(p_j)\log p_j}{p_j \log y} =\left(\sum_{p\leq y}\tfrac{\gamma(p)\log p}{p\log y}\right)^{k-1} = \left(1 + O\left(\tfrac{1}{\log y}\right)\right)^{k-1} = O_k(1).
    \end{align*}
    Et ainsi
    \begin{align*}
        S &= \kappa\sum_{p_1, \dotsc, p_{k-1} }  \left( \int_{t_k \in I(p_1, \dotsc, p_{k-1})} f_{p_1, \dotsc, p_{k-1}}(t_k) \dd t_k\right)\prod_{j=1}^{k-1}\tfrac{\gamma(p_j)\log p_j}{p_j \log y} + O_k\left(\tfrac{C}{\log y}\right).
    \end{align*}
    Notons maintenant pour $t_k>0$
  $$ V_{t_k} = \{(t_1, \dots, t_{k-1}): (t_1, \dotsc, t_k)\in U\}, $$
  $$ f_{t_k}(t_1, \dotsc, t_{k-1}) := f(t_1, \dotsc, t_{k}). $$
  L'ensemble~$V_{t_k}$ est convexe et inclus dans~$[0, 1]^{k-1}$, et la fonction~$f_{t_k}$ satisfait bien
  $$ \|f_{t_k}\|_{\infty} + \sum_{j=1}^{k-1} \Big\|\frac{\partial f_{t_k}}{\partial t_j}\Big\|_{\infty} \leq C. $$
  Le terme principal de $S$ ci-dessus se réécrit
  $$ \kappa\int_{t_k\in[0, 1]} \ssum{p_1, \dotsc, p_{k-1} \\ \left(\frac{\log p_j}{\log y}\right)\in V_{t_k}}f_{t_k}\left(\frac{\log p_1}{\log y},\ldots,\frac{\log p_{k-1}}{\log y}\right)\prod_{j=1}^{k-1} \frac{\gamma(p_j)\log p_j}{p_j\log y}  \dd t_k. $$
  En vertu de ce qui précède, nous pouvons appliquer l'hypothèse de récurrence uniformément en~$t_k$ et intégrer le terme d'erreur obtenu, ce qui fournit
  \begin{align*}
    & \kappa\int_{t_k\in[0, 1]} \ssum{p_1, \dotsc, p_{k-1} \\ (\log p_j / \log y)\in V_{t_k}}f_{t_k}\left(\frac{\log p_1}{\log y},\ldots,\frac{\log p_{k-1}}{\log y}\right)\prod_{j=1}^{k-1} \frac{\gamma(p_j)\log p_j}{p_j\log y}  \dd t_k \\
    ={}& \kappa\int_{t_k\in [0, 1]} \kappa^{k-1}\idotsint_{V_{t_k}} f_{t_k}(t_1, \dotsc, t_{k-1}) \dd t_1 \dotsc \dd t_k + \int_0^1O_k\left(\frac{C}{\log y}\right)\dd t_k \\
    ={}& \kappa^{k}\int_U f(\bm{t})\dd \bm{t} + O_k\left(\frac{C}{\log y}\right).
  \end{align*}

\end{proof}

\begin{lem}\label{lem:casseomega}
    Pour $k\geq 1$ et tous réels $a$ et $b$ tels que $a\leq b<u$, sous la condition~\eqref{cond2},
    \begin{align*}
        \ssum{p_1,\ldots,p_k\leq y\\y^a < p_1\cdots p_k < y^b}&\frac{Q(y^u;p_1,\ldots,p_k)\gamma(p_1\cdots p_k)}{p_1\cdots p_k} =\\
        &\ssum{p_1,\ldots,p_k\leq y\\y^a < p_1\cdots p_k < y^b}\frac{Q(y^u;p_1,\ldots,p_k)\gamma(p_1)\cdots \gamma(p_k)}{p_1\cdots p_k} + O_{k}\left((u-b)^{-k}(\log y)^{-2}\right).
    \end{align*}
\end{lem}

\begin{remark}
    Pour un entier naturel non-nul $a$ fixé, la même relation est valable avec le même terme d'erreur mais avec une constante implicite dépendant éventuellement de $a$ en ajoutant la condition ``$(p_1\cdots p_k,a)=1$'' aux deux sommes.
\end{remark}

\begin{proof}
    Notons $S$ la quantité
 $$S := \ssum{p_1,\ldots,p_k\leq y\\y^a < p_1\cdots p_k < y^b}\frac{Q(y^u;p_1,\ldots,p_k)h(y;p_1,\ldots, p_k)}{p_1\cdots p_k}$$
 où $h(y;p_1,\ldots,p_{k}) := \gamma(p_1\cdots p_{k}) -\gamma(p_1)\cdots\gamma(p_k)$. Comme $\gamma$ est multiplicative, $h$ est nulle lorsque les $p_i$ sont deux à deux distincts. Nous pouvons donc supposer $p_1=p_2$ au prix d'un facteur $k!$. D'autre part, la condition~\eqref{cond2} fournit 
 $$h(y;p_1,\cdots,p_k) \ll c_1^k$$
 pour une certaine constante $c_1>0$. Ainsi,
\begin{align*}
    S&\ll_k \ssum{p_1,p_3,\cdots,p_k\leq y\\y^a<p_1p_3\cdots p_k\leq y^b}c_1^k\frac{Q(y^u;p_1,p_1,p_3,\cdots,p_k)}{p_1^2p_3\cdots p_k}\\
    &\ll_k \frac{1}{(u-b)^k(\log y)^k}\left(\sum_{p\leq y}\frac{\log p}{p^2}\right)\left(\sum_{p\leq y}\frac{\log p}{p}\right)^{k-2}\\
    &\ll_{k}(u-b)^{-k}(\log y)^{-2}.
\end{align*}
 
\end{proof}

\begin{lem}\label{lem:conda}
Soit $a\geq 1$. Pour tout $(t,s)$ vérifiant $0<t<s<u$, sous la condition~\eqref{cond2}
$$\ssum{p_1,\ldots,p_{k}\leq y\\y^{t}<p_1\cdots p_{k} \leq y^s\\(p_1\cdots p_k,a)\neq 1}\frac{Q(y^u,p_1,\ldots,p_{k})\gamma(p_1)\cdots \gamma(p_k)}{p_1\cdots p_k} = O_{k,a}\left((u-s)^{-k}(\log y)^{-1}\right).
$$
\end{lem}

\begin{proof}
    Notons $S$ la somme du lemme, $D_a$ l'ensemble des diviseurs premiers de $a$ et $r$ son cardinal. D'après la condition~\eqref{cond2}, 
    $$S \ll_k \ssum{p_1,\ldots,p_{k}\leq y\\y^{t}<p_1\cdots p_{k} \leq y^s\\(p_1\cdots p_k,a)\neq 1}\frac{Q(y^u,p_1,\ldots,p_{k})}{p_1\cdots p_k}.$$
    Notons $\mathcal{P}_k$ l'ensemble des parties de $\{1,\ldots,k\}$. La somme ci-dessus se réécrit
    \begin{align*}
        \sum_{\substack{J\in \mathcal{P}_k\\J\neq \emptyset}}\sum_{\substack{p_1,\ldots, p_k< y\\y^t<p_1\cdots p_k\leq y^s\\j\in J\Leftrightarrow p_j\in D_a}}\frac{Q(y^u;p_1,\ldots,p_k)}{p_1\cdots p_k}
        \ll_{k} ((u-s)\log y)^{-k}\sum_{\substack{J\in \mathcal{P}_k\\J\neq \emptyset}}\sum_{\substack{p_1,\ldots, p_k< y\\y^t<p_1\cdots p_k\leq y^s\\j\in J\Leftrightarrow p_j\in D_a}}\frac{\log p_1}{p_1}\cdots\frac{\log p_k}{p_k}.
    \end{align*}
    À ce stade, la condition ``$y^t<p_1\cdots p_k\leq y^s$'' peut être oubliée afin de rendre la somme interne séparable. Par ailleurs, dans la somme interne, exactement~$|J|$ nombres premiers parmi $p_1,\ldots,p_k$ sont des diviseurs de $a$. Nous avons donc
    \begin{align*}
        S&\ll_{k} ((u-s)\log y)^{-k}\sum_{\substack{J\in \mathcal{P}_k\\J\neq \emptyset}}\left(\prod_{j\in J}\ssum{p\leq y\\p|a}\frac{\log p}{p}\right)\left(\prod_{j\in \{1,\ldots,k\}\backslash J}\sum_{p\leq y}\frac{\log p}{p}\right)\\
        &\ll_{k} ((u-s)\log y)^{-k}\sum_{\substack{J\in \mathcal{P}_k\\J\neq \emptyset}}\left(\sum_{p|a}\frac{\log p}{p}\right)^{|J|}\left(\sum_{p\leq y}\frac{\log p}{p}\right)^{k-|J|}\\
        &\ll_{k,a} ((u-s)\log y)^{-k}\sum_{\substack{J\in \mathcal{P}_k\\J\neq \emptyset}}(\log y+O(1))^{k-|J|}\\
        &\ll_{k,a} ((u-s)\log y)^{-k}\sum_{i=1}^k(\log y+O(1))^{k-i}\binom{k}{i} \\
        &\ll_{k,a}(u-s)^{-k}(\log y)^{-1}. 
    \end{align*}
\end{proof}

\subsection{Démonstration du Théorème ~\ref{theo:principal}}\label{sec:demotheoprincipal}

L'ingrédient clé de notre méthode repose sur une idée qui remonte à Tchebychev et qui fut utilisée par Hildebrand \cite{Hil86} afin d'étudier le comportement asymptotique de $\Psi(x,y)$. Il s'agit d'évaluer de deux manières différentes la somme 
$$\ssum{n\leq x\\P^+(n)\leq y}\log n.$$
D'une part, le poids ``$\log n$'' se décompose grâce à la relation $\log = \Lambda*1$ et d'autre part, une transformation d'Abel permet de relier la somme ci-dessus à une intégrale. Cela fournit l'identité de Hildebrand: 
$$\Psi(x,y)\log x = \int_1^x\frac{\Psi(t,y)}{t}\,\dd t + \sum_{d\in S(x,y)}\Lambda(d)\Psi\left(\tfrac{x}{d},y\right).$$
Nous généralisons ici cette identité afin d'étudier $\Psi_{\mathcal{A}}(y)$ au lieu de $\Psi(x,y)$, ce qui mène au résultat suivant.
\begin{prop}
Pour tout $x\geq 1$,

\begin{equation}\label{eq:HilA}
    \Psi_{\mathcal{A}}(y)\log(x) = \int_1^x\Psi_{\mathcal{A}_{\leq t}}(y)\,\frac{\dd t}{t} + \ssum{d\leq x\\ P^+(d)\leq y}\Lambda(d)\Psi_{\mathcal{A}}(y,d)
\end{equation}    
    où $\mathcal{A}_{\leq t} := \mathcal{A}\cap [1,t]$.
\end{prop}
\begin{proof}
Notons 
$$S := \sum_{n\in S_{\mathcal{A}}(y)}\log n$$
où $S_{\mathcal{A}}(y) = \{n\in\mathcal{A}:P^+(n)\leq y\}$. D'une part, par sommation d'Abel,
$$S = \Psi_{\mathcal{A}}(y)\log(x) - \int_1^x\Psi_{\mathcal{A}_{\leq t}}(y)\frac{\mathrm{d}t}{t}$$
et d'autre part la relation $\log = 1*\Lambda$ fournit $$S=\sum_{n\in S_\mathcal{A}}\sum_{d|n}\Lambda(d)$$
et le résultat s'obtient en permutant les sommes.
\end{proof}

\begin{cor}
\begin{equation}\label{eq:majPsiPsid}
    \Psi_{\mathcal{A}}(y)\log x \geq \ssum{p\leq y}\Psi_{\mathcal{A}}(y,p)\log p
\end{equation}

\end{cor}

\begin{proof}
    L'intégrant de~\eqref{eq:HilA} étant positif, il suffit de minorer l'intégrale par $0$. Quant à la somme, les termes étant également positifs, nous minorons en ne conservant que les indices premiers.
\end{proof}

Nous sommes maintenant ramenés à minorer le membre de droite dans~\eqref{eq:majPsiPsid}. Dans cette optique, pour tout $n\geq 1$ posons:
$$T(y,n) := \ssum{p\leq y\\ np\leq y^u}\Psi_{\mathcal{A}}(y,pn)\log p$$
et
$$R(y,n) := \ssum{p\leq y\\ y^{u-1}<np\leq y^u}A_{pn}\log p.$$
La proposition suivante est la pierre angulaire de notre méthode. En effet, elle incite à effectuer une récurrence sur $T(y,n)$, ce qui fait l'objet de la Proposition~\ref{prop:TT1T2R}. Entre cette dernière et le Théorème~\ref{theo:principal}, il n'y a essentiellement que des calculs permettant de se ramener aux fonctions~$\mu_k$ introduites dans la partie $\ref{sec:muk}$.

\begin{prop}\label{prop:TTR}
Pour tout $n\geq 1$ $y$-friable, 
    $$T(y,n) \geq  
        \ssum{p\leq y\\ np\leq y^{u-1}}\frac{\log p}{\log\left(\frac{y^u}{pn}\right)}T(y,pn) + R(y,n).$$
\end{prop}

\begin{proof}
Commençons par séparer la somme définissant $T$ comme suit:
    \begin{align*}
        T(y,n) &\geq \ssum{p\leq y\\ np\leq y^{u-1}}\Psi(y^u,pn)\log p + \ssum{p\leq y\\ y^{u-1}<np\leq y^{u}}\Psi(y^u,pn)\log p\\
        &=: S_1 + S_2.
    \end{align*}
    Dans la somme $S_2$, le Lemme~\ref{lem:PsiAd} implique que $\Psi(y^u,pn) = A_{pn}$. Cela entraîne $S_2 = R(y,n)$. Pour minorer $S_1$, il suffit de démontrer 
    $$T(y,m)\leq \Psi(y^u,m)\log\left(\frac{y^u}{m}\right)$$
    pour tout $m\geq 1$ et d'appliquer cette inégalité à $m=pn$.
    Notant $\chi_y$ la fonction indicatrice des entiers $y$-friables, nous majorons $T(y,m)$ comme suit:
    \begin{align*}
        T(y,m) &\leq \sum_{p\leq\frac{y^u}{m}}\Psi_{\mathcal{A}}(y,pm)\log p\\
        &= \sum_{p\leq\frac{y^u}{m}}\log p\sum_{k\leq\frac{y^u}{pm}}\1_{\mathcal{A}}(kpm)\chi_y(kpm)\\
        &= \ssum{k\leq\frac{y^u}{m}}\1_{\mathcal{A}}(km)\chi_y(km)\sum_{p|k}\log p\\
        &\leq \ssum{k\leq\frac{y^u}{m}}\1_{\mathcal{A}}(km)\chi_y(km)\log k\\
        &\leq \log\left(\frac{y^u}{m}\right)\ssum{k\leq\frac{y^u}{m}}\1_{\mathcal{A}}(km)\chi_y(km)\\
        &= \Psi_{\mathcal{A}}(y^u,m)\log\left(\frac{y^u}{m}\right).
    \end{align*}
\end{proof}

\begin{prop}\label{prop:TT1T2R}
    Pour tout $k\geq 1$, 
    $$T(y,1) \geq R(y,1)+ \sum_{i=1}^{k}\ssum{p_1,\ldots,p_i\leq y\\p_1\cdots p_i \leq y^{u-1}}Q(y^u,p_1,\ldots,p_i)R(y,p_1\cdots p_i)$$
\end{prop}

\begin{proof}
Montrons par récurrence le résultat plus fort suivant:\\
    \begin{equation}\label{eq:T2T3}
        T(y,1) \geq R(y,1) + T_1(y,k)  + T_2(y,k)
    \end{equation}
    où
    $$T_1(y,k) := \ssum{p_1,\ldots,p_k\leq y\\p_1\cdots p_k \leq y^{u-1}}Q(y^u,p_1,\ldots,p_k)T(y,p_1\cdots p_k)$$
    et 
    $$T_2(y,k) := \sum_{i=1}^{k-1}\ssum{p_1,\ldots,p_i\leq y\\p_1\cdots p_i \leq y^{u-1}}Q(y^u,p_1,\ldots,p_i)R(y,p_1\cdots p_i).$$
    Le cas $k=1$ découle de la Proposition~\ref{prop:TTR} avec $n=1$. Soit $k\geq 1$ tel que la proposition est vraie au rang $k$. La minoration de $T$ donnée dans la Proposition~\ref{prop:TTR} fournit
    \begin{align*}
        T_1(y,k) &= \ssum{p_1,\ldots,p_{k}\leq y\\p_1\cdots p_{k} \leq y^{u-1}}Q(y^u,p_1,\ldots,p_{k})T(y,p_1\cdots p_{k})\\
        &\geq \ssum{p_1,\ldots,p_k\leq y\\ p_1\cdots p_k \leq y^{u-1}}Q(y^u,p_1,\ldots,p_k)\ssum{p_{k+1}\leq y\\p_1\cdots p_{k+1}\leq y^{u-1}}\tfrac{\log p_{k+1}}{\log\left(\tfrac{y^u}{p_1\cdots p_{k+1}}\right)}T(y,p_1\cdots p_{k+1})\\
        &\quad + \ssum{p_1,\ldots,p_k\leq y\\ p_1\cdots p_k \leq y^{u-1}}Q(x,p_1,\ldots,p_k)R(y,p_1\cdots p_k)\\
        &= T_1(y,k+1) + T_2(y,k+1) - T_2(y,k)
    \end{align*}
    Il n'y a plus qu'à réinjecter dans~\eqref{eq:T2T3} pour constater que la proposition est vraie au rang~$k~+~1$.
\end{proof}

Le lemme suivant permet de minorer $R(y,p_1\cdots p_i)$ dans la formule de la Proposition~\ref{prop:TT1T2R}.

\begin{lem}\label{lem:majT}
    Pour tout $n\geq 1$ $y$-friable, sous les conditions~\eqref{cond1} et~\eqref{cond2}
    $$ R(y,n)= \kappa\frac{\gamma(n)}{n}X\log y\int_0^1\1_{]y^{u-1},y^{u\theta}]}(ny^v)\dd v + O_k\left(\tfrac{X}{n}\right)+ \ssum{p\leq y\\ np\leq y^{u\theta}}r_{np}\log p $$
    où $k:=\Omega(n)$.
\end{lem}

\begin{proof}
    Il suffit de remarquer que par définition,  $A_{np} =\frac{\gamma(np)}{np}X + r_{np}$, donc 
    \begin{align*}
        R(y,n)&= \frac{X}{n}\ssum{p\leq y\\ y^{u-1}<np\leq y^{\theta u}}\frac{\log p}{p}\gamma(np) + \ssum{p\leq y\\ y^{u-1}<np\leq y^{\theta u}}r_{np}\log p.
    \end{align*}
    Notons 
    $$S := \ssum{p\leq y\\ y^{u-1}<np\leq y^{\theta u}}\frac{\log p}{p}\gamma(np).$$
    Nous avons donc 
    \begin{align*}
        S &= \gamma(n)\ssum{p\leq y\\ y^{u-1}<np\leq y^{\theta u}}\frac{\log p}{p}\gamma(p) + \ssum{p\leq y\\ y^{u-1}<np\leq y^{\theta u}}\{\gamma(np)-\gamma(n)\gamma(p)\}\frac{\log p}{p}.
    \end{align*}
    Notons $S_M$ la première somme et $S_E$ la seconde. Comme $\gamma$ est multiplicative, les termes de $S_E$ sont nuls dès que $p$ ne divise pas $n$, ce qui permet d'écrire:
    \begin{align*}
        |S_E| &\leq \sum_{p|n} |\gamma(np)-\gamma(n)\gamma(p)|\frac{\log p}{p}\ll c^{\omega(n)}\ll_k 1
    \end{align*}
    pour une certaine constante $c$ provenant de la condition~\eqref{cond2}. En ce qui concerne $S_M$, nous appliquons le Lemme~\ref{lem:general} avec $k=1$ et $f$ la fonction constante valant $1$ sur $U=\left[u-1-\frac{\log n}{\log y},\theta u-\frac{\log n}{\log y}\right]\cap[0,1]$ , ce qui fournit bien
    \begin{align*}
        S_M = \kappa\log y\int_0^1\1_{\left]y^{u-1},y^{u\theta}\right]}(ny^v)\dd v + O(1).
    \end{align*}
\end{proof}

\begin{proof}[Démonstration du Théorème~\ref{theo:principal}]
    Plaçons nous tout d'abord dans le cas où l'hypothèse~\eqref{cond3} est vérifiée pour un certain $\theta\in\mathclose{]}1-\frac{1}{u},1\mathopen{[}$. D'après la Proposition~\ref{prop:TT1T2R} et le Lemme~\ref{lem:majT}, sous les conditions~\eqref{cond1} et~\eqref{cond2},
    \begin{align}\label{eq:sig123}
        T(y,1) &\geq R(y,1) + \Sigma_1 + \Sigma_2 + \Sigma_3 
    \end{align}
    où 
    \begin{align*}
        \Sigma_1 &= \kappa X\log y \sum_{i=1}^k\ssum{p_1,\ldots,p_{i}\leq y\\p_1\cdots p_{i} \leq y^{u-1}}\tfrac{Q(y^u,p_1,\ldots,p_{i})\gamma_y(p_1\cdots p_i)}{p_1\cdots p_i}\int_0^1\1_{\left]y^{u-1},y^{\theta u}\right]}(p_1\cdots p_iy^v)\dd v,\\
        \Sigma_2 &= X\sum_{i=1}^kO_i\left(\ssum{p_1,\ldots,p_{i}\leq y\\p_1\cdots p_{i} \leq y^{u-1}}\frac{Q(y^u,p_1,\ldots,p_{i})}{p_1\cdots p_i}\right),\\
        \Sigma_3 &= \sum_{i=1}^k\ssum{p_1,\ldots,p_{i}\leq y\\p_1\cdots p_{i} \leq y^{u-1}}Q(y^u,p_1,\ldots,p_{i})\ssum{p\leq y\\ p_1\cdots p_ip\leq y^{\theta u}}r_{p_1\cdots p_ip}\log p.
    \end{align*}
    Les termes $R(y,1)$ et $\Sigma_1$ fournissent le terme principal tandis que $\Sigma_2$ et~$\Sigma_3$ sont les termes d'erreurs.
    Nous majorons d'abord $\Sigma_2$, pour cela nous observons que
    \begin{equation}\label{eq:Qpourlem}
        Q(y^u;p_1,\ldots,p_i) = f_i\left(u;\frac{\log p_1}{\log y},\ldots,\frac{\log p_i}{\log y}\right)\prod_{j=1}^i\frac{\log p_j}{\log y}
    \end{equation}
    où $f_i$ est définie en \eqref{eq:fk}. Nous appliquons le Lemme~\ref{lem:general} avec $$U = \{\mathbf{t}\in[0,1]^i~:~t_1+\cdots + t_i\leq u-1\}\quad\text{et}\quad f=f_i(u;.),$$ ce qui fournit:
    \begin{align*}
        \ssum{p_1,\ldots,p_{i}\leq y\\p_1\cdots p_{i} \leq y^{u-1}}\frac{Q(y^u,p_1,\ldots,p_{i})}{p_1\cdots p_i} &\ll_i \int_Uf_i(u;\mathbf{t})\dd \mathbf{t}\ll_i \mathrm{Vol}(U)\ll_i 1,
    \end{align*}
    la deuxième majoration venant du fait que $f_i(u;\mathbf{t})\leq 1$ pour tout $\mathbf{t}\in U$. Finalement, nous trouvons
    \begin{equation}\label{eq:majSigma2}
       \Sigma_2 =  O_k(X) 
    \end{equation}
    Ensuite, $\Sigma_3$ se réécrit
    \begin{equation}\label{eq/sigma3}
        \Sigma_3 = \sum_{i=1}^{k+1}\ssum{p_1,\ldots,p_{i}\leq y\\p_1\cdots p_{i} \leq y^{\theta u}\\p_1\cdots p_{i-1} \leq y^{u-1}}Q(y^u,p_1,\ldots,p_{i-1})r_{p_1\cdots p_i}\log p_i
    \end{equation}
    ce qui d'après le Lemme~\ref{lem:minQE} entraîne 
    \begin{equation}\label{eq:majSigma3}
        \Sigma_3 = O_{k,\theta}(X(\log y)^{1-\delta})
    \end{equation}
    Il reste à estimer $\Sigma_1$ et $R(y,1)$:
    \begin{equation}\label{eq:eqsigma1}
        \Sigma_1 = \kappa X\log y \sum_{i=1}^k\int_0^{1}\ssum{p_1,\ldots,p_{i}\leq y\\y^{u-v-1}<p_1\cdots p_{i} \leq\min(y^{u-1},y^{u\theta-v})}\frac{Q(y^u,p_1,\ldots,p_{i})\gamma(p_1\cdots p_i)}{p_1\cdots p_i}\dd v.
    \end{equation}
    Le Lemme~\ref{lem:casseomega} nous permet de remplacer $\gamma(p_1\cdots p_i)$ par $\gamma(p_1)\cdots \gamma(p_i)$ dans la somme sous l'intégrale au prix d'un
    $$O_{i}\left(\max(1,v+u(1-\theta))^{-i}(\log y)^{-2}\right) = O_{i}\left((\log y)^{-2}\right),
    $$
    ce qui, grâce à \eqref{eq:Qpourlem}, fournit
    $$\ssum{p_1, \dotsc, p_i\in \mathcal{P} \\ \left(\frac{\log p_j}{\log y}\right)_{1\leq j\leq i}\in U(v)} f_i\left(u;\frac{\log p_1}{\log y},\ldots,\frac{\log p_i}{\log y}\right)\prod\limits_j\frac{\gamma(p_j)\log p_j}{p_j\log y}+ O_i((\log y)^{-2})$$
    avec 
    $$U(v) := \{\bm{t}\in[0,1]^i:t_1 + \cdots + t_{i} < u-1 < t_1 + \cdots + t_i + v < \theta u\}$$
    Le Lemme~\ref{lem:general} permet donc d'estimer la somme ci-dessus à l'aide d'intégrales:
    \begin{align}\label{eq:sumintavecconda}
        \kappa^i\int_{U(v)} f_i(u;t_1,\ldots,t_i)\dd t_1\cdots \dd t_i + O_i((\log y)^{-1}).
    \end{align}
    En combinant \eqref{eq:eqsigma1} et \eqref{eq:sumintavecconda}, nous obtenons
    \begin{align*}
        \Sigma_1 &= \kappa X \log y \sum_{i=1}^k\int_0^1\kappa^i\int_{U(v)} f_i(u;t_1,\ldots,t_i)\dd t_1\cdots \dd t_i\dd v + O_k(X)\\
        &= \kappa X \log y \sum_{i=1}^k\kappa^i\int_{V} f_i(u;t_1,\ldots,t_i)\dd t_1\cdots \dd t_i\dd v + O_k(X)
    \end{align*}
    où 
    \begin{align*}
        V &:= \{\mathbf{t}\in\oo{0,1}^{i+1}~:~t_1 +\cdots + t_i < u-1 < t_1 + \cdots + t_{i+1} < \theta u\}\\
        &= V_{i+1}(u,(1-\theta) u)
    \end{align*}
    d'après l'expression de $V_{i+1}$ dans la définition \ref{def:Vketmuk}. Nous avons donc
    \begin{equation}\label{eq:Sigma1final}
        \Sigma_1 = \kappa X \log x \sum_{i=2}^{k+1}\kappa^i\mu_i(u,(1-\theta)u) + O_k(X)\\
    \end{equation}
    d'après l'expression de $\mu_i$ dans la définition \ref{def:Vketmuk}. Combinant \eqref{eq:sig123}, \eqref{eq:majSigma2}, \eqref{eq:majSigma3} et \eqref{eq:Sigma1final}, nous avons montré 
    $$T(y,1) \geq R(y,1) + X\log x\sum_{i=2}^{k+1}\kappa^i\mu_i(u,(1-\theta)u) + O_{k,\theta}(X(\log y)^{1-\delta}).$$
    Par ailleurs, le lemme \ref{lem:majT} fournit
    $$R(y,1) = \kappa X \mu_1(u,(1-\theta)u)\log x + O(X(\log y)^{1-\delta}),$$
    ce qui permet d'écrire
    $$T(y,1) \geq X\log x\sum_{i=1}^{k+1}\kappa^i\mu_i(u,(1-\theta)u) + O_{k,\theta}(X(\log y)^{1-\delta}),$$
    puis avec~\eqref{eq:majPsiPsid},
    \begin{align*}
        \Psi_{\mathcal{A}}(y) &\geq X\sum_{i=1}^{k+1}\kappa^i\mu_i(u,(1-\theta)u)+ O_{k,\theta}(X(\log y)^{-\delta}).
    \end{align*}
    Choisissant $k$ convenablement par rapport à $x$, nous obtenons 
    \begin{align*}
        \Psi_{\mathcal{A}}(y) &\geq X\varrho_{\theta,\kappa}(u) + o(X)
    \end{align*}
    lorsque $x\to\infty$, par définition de $\varrho_{\theta,\kappa}$ en \eqref{eq:rhotheta}.\\

    Il reste maintenant à montrer la remarque~\ref{rem:conda1}, c'est-à-dire que le Théorème~\ref{theo:principal} reste vrai lorsque la condition~\eqref{cond3} est remplacée par la condition~\eqref{cond4}.
    Dans ce cas, la minoration~\eqref{eq:sig123} est encore valable en remplaçant respectivement $\Sigma_1$, $\Sigma_2$ et $\Sigma_3$ par $\Sigma_1'$, $\Sigma_2'$ et $\Sigma_3'$ dans lesquelles la condition~``$(p_1\cdots p_k,a)=1$'' est ajoutée  dans chaque somme concernée.\\
        La majoration $\Sigma_3'=O_{k,\theta,a}(X(\log y)^{1-\delta})$ est garantie par~\eqref{eq:sumEQ2} appliquée à~\eqref{eq/sigma3}. La majoration~$\Sigma_2'=O_k(X)$ découle directement de $\Sigma_2=O_k(X)$. En ce qui concerne $\Sigma_1'$, la relation $$\Sigma_1'=X\log y \sum_{i=1}^{k+1}\kappa^i\mu_i(u,(1-\theta)u) + O_k(X)$$
    est valable grâce au Lemme~\ref{lem:conda} appliqué à~\eqref{eq:eqsigma1}. En conclusion, le Théorème~\ref{theo:principal} reste vrai si la condition~\eqref{cond3} est remplacée par la condition~\eqref{cond4}.
\end{proof}

\section{Application aux valeurs friables de $X^2+1$}\label{sec:xcarreplus1}

Dans cette partie, nous démontrons le Théorème~\ref{theo:ncarre}. Considérons le polynôme $F(X) = X^2+1$. Nous étudions le comportement asymptotique de
$$\Psi_{F}(x,y) := \#\{n\leq x:P^+(n^2+1)\leq y\}.$$ 
Afin de se replacer dans le cadre du théorème~\ref{theo:principal}, nous choisissons $$\mathcal{A}:=\{n^2+1:n\leq x\},$$ si bien que $\Psi_F(x,y) = \Psi_\mathcal{A}(y)$.\\
Notons  $x':=\sup\mathcal{A}=x^2+O(1)$ et $u':=\frac{\log x'}{\log y}$.
Enfin, pour $d\geq 1$, 
$$\gamma(d) := \#\{0\leq\alpha<d:\alpha^2\equiv -1\bmod d\}.$$
\begin{prop}\label{eq:hyp2ncarre}
    La fonction $\gamma$ est multiplicative et vérifie les conditions~\eqref{cond1} avec $\kappa=1$ et~\eqref{cond2}. Elle vérifie aussi la condition~\eqref{cond3} avec le choix $X=x$ et $\theta=\frac{1}{2} - \frac{1}{\log\log x}$.
\end{prop}

\begin{proof}
    La démonstration de la multiplicativité est similaire à celle du théorème des restes chinois.
    Pour les conditions~\eqref{cond1} et~\eqref{cond2}, il s'agit du point $(5)$ de l'article de Hooley \cite{Hoo67} dans lequel~$\varrho$ joue le rôle de notre $\gamma$. Pour la condition~\eqref{cond3}, nous verrons dans la section suivante un cas plus général, néanmoins nous donnons ici le détail dans ce cas particulier. En effet d'après le point $(4)$ de \cite{Hoo67}, 
    $$A_d = \frac{\gamma(d)}{d}x + O(\gamma(d))$$
d'où
\begin{align*}
    \sum_{d\leq y^{u'\theta}}\left| A_d-\frac{\gamma(d)}{d} x\right|\ll \sum_{d\leq y^{u'\theta}}\gamma(d)\ll \sum_{d\leq y^{u'\theta}}c^{\omega(d)} \ll x^{2\theta}(\log y)^{c-1}
\end{align*}
pour un $c>0$ dépendant seulement de $\gamma$ (en fait, $c=2$ convient).
La deuxième minoration découle de la condition~\eqref{cond1} et de la multiplicativité de $\gamma$. La troisième vient de la majoration~$\sum_{n\leq x}c^{\omega(n)} \ll x(\log x)^{c-1}$ (voir Tome I, Chapitre 6, Théorème $6.1$ de \cite{Ten15}).
Il suffit ensuite de remplacer $\theta$ par~$\frac{1}{2}-\frac{1}{\log\log x}$.
    
\end{proof}

\begin{proof}[Démonstration du Théorème~\ref{theo:ncarre}]
    D'après ce qui précède, le Théorème~\ref{theo:principal} s'applique. Il faut cependant prendre soin de ne pas confondre $x$ avec $x'$, en effet $x$ ne désigne plus la borne supérieure de $\mathcal{A}$: c'est $x'$ qui joue ce rôle ici. Ainsi, pour $\theta =\frac{1}{2}-\frac{1}{\log\log x}$ et $1\leq u' \leq \frac{1}{1-\theta}$, lorsque~$x\to\infty$,
    $$\Psi_F(x,y) = \Psi_{\mathcal{A}}(y)\geq \varrho_{\theta,1}(u')x + o(x)$$
    puis~\eqref{eq:theta12} fournit 
    $$\Psi_F(x,y)\geq \left(1-\left(\frac{1}{2}-\frac{1}{\log\log x}\right)u'\right)x + o(x),$$
    valable pour tout $1\leq u \leq 2$.
    Le théorème découle ensuite de $u' = 2u + o(1)$.
    
\end{proof}

\section{Application aux valeurs friables des polynômes irréductibles de $\Z[X]$}\label{sec:poly}

Dans cette partie, nous démontrons le Théorème~\ref{theo:Fn}.\\
Soit $\varepsilon>0$ et~$F\in\Z[X]$ irréductible, de degré $g\geq 2$. Prenons $$\mathcal{A}:=\{|F(n)|:n\leq x,\,F(n)\neq 0\},$$ si bien que $\Psi_{F}(x,y) = \Psi_{\mathcal{A}}(y)$.

Notons $x' := \sup\mathcal{A}$ et $u' := \frac{\log x'}{\log y} = gu+O\left(\frac{1}{\log x}\right)$.
Enfin, 
$$\gamma_{F}(d) := \#\{x\bmod{d}:F(x) = 0\bmod{d}\}.$$

\begin{prop}\label{prop:condpoly}
    La fonction $\gamma_{F}$ est multiplicative et vérifie les conditions~\eqref{cond1} pour $\kappa=1$ et~\eqref{cond2}. Elle vérifie aussi la condition~\eqref{cond3} avec le choix $X=x$ et $\theta=\frac{1}{g} - \frac{1}{\log\log x}$.
\end{prop}

\begin{proof}
    De même que précédemment, la démonstration de la multiplicativité est similaire à celle du théorème des restes chinois. Les conditions~\eqref{cond1} avec $\kappa=1$ et~\eqref{cond2} sont des résultats connus~(voir \cite{Nag21}). Pour la condition~\eqref{cond3}, le lecteur pourra se référer à \cite{HR74}, I.3.14, chapitre IX.5.

\end{proof}

\begin{proof}[Démonstration du Théorème~\ref{theo:Fn}]

D'après ce qui précède, le Théorème~\ref{theo:principal} s'applique avec~$\theta =\frac{1}{g}-\frac{1}{\log\log x}$ et $1\leq u'<\frac{g}{g-1}$, 
$$\Psi_F(x,y) = \Psi_\mathcal{A}(y) \geq \varrho_{\theta,g}(u')x + o(x)$$
lorsque $x\to\infty$, ce qui d'après~\eqref{eq:theta12} entraîne pour $x^{g-1}<y<x^{g}$:
$$\Psi_F(x,y) \geq \{1+u-gu+o(1)\}x.$$

\end{proof}

\section{Application aux formes binaires à coefficients entiers}\label{sec:formebinaire}

Dans cette partie, nous démontrons le Théorème~\ref{theo:Formebinaire}. Soit $F(X,Y)$ une forme binaire à coefficients entiers, irréductible, de degré $t\geq 2$. Posons $$\mathcal{A}:=\{|F(a,b)|:1\leq a,b \leq x,\,(a,b)=1,\,F(a,b)\neq 0\}.$$ En particulier, $\Psi_F(x,y) \geq \Psi_{\mathcal{A}}(y)$. Notons encore $x' := \sup \mathcal{A}$ et $u' := \frac{\log x'}{\log y}$. En particulier,~$u' = ut + o(1)$ lorsque $x\to\infty$. De plus, si $d\geq 1$, alors 
\begin{align*}
    A_d &= \#\{1\leq a,b\leq x:(a,b)=1,\,F(a,b)\neq 0\text{ et }d|F(a,b)\}.
\end{align*}
Conservant la notation de Greaves (\cite{Gre71}, (2.2.1)), notons pour $d\geq 1$
$$\gamma^*(d) := \frac{1}{d}\#\{1\leq\alpha,\beta\leq d:F(\alpha,\beta)\equiv 0\pmod{d}\text{ et }(\alpha,\beta,d)=1\}.$$

\begin{prop}
    La fonction $\gamma^*$ est multiplicative et vérifie les conditions~\eqref{cond1} avec $\kappa=1$ et~\eqref{cond2}. Elle vérifie aussi la condition~\eqref{cond3} avec le choix $X=x^2$ et $\theta = \frac{2}{t+\varepsilon_x}(1 - \frac{1}{\log\log x})$ où $\varepsilon_x := u'/u-~t \tendvers{x}{+\infty} 0$.
\end{prop}

\begin{proof}
    Dans \cite{Gre71}, la notation $\rho$ correspond à notre $\gamma$ de la partie précédente car~$F(X,1)$ est un polynôme irréductible de $\Z[X]$. Donc le point (2.2.2) \emph{ibid.} et notre Proposition~\ref{prop:condpoly} fournissent
\begin{equation*}
    \sum_{n\leq x}\frac{\log p}{p}\gamma^*(p) = \log x + O(1)
\end{equation*}
et la condition~\eqref{cond1} est vérifiée. La condition~\eqref{cond2} correspond au point (2.2.3) dans \cite{Gre71}. Pour montrer que la condition~\eqref{cond3} est vérifiée, nous avons besoin d'un résultat de Daniel~(\cite{Dan99}, Lemme~3.2): avec la notation $\nu_t := t(1+2t)^{t+1}$ et en choisissant $M=4x$ et $Q=y^{u'\theta}$, il vient
\begin{align*}
    \sum_{d\leq y^{u'\theta}}\left|A_d-\frac{\gamma^*(d)}{d}x^2\right| &\ll x y^{u'\theta/2}(\log 2y^{u'\theta})^{\nu_t} + y^{u'\theta}(\log 2y^{u'\theta})^{2t}\\
    &=  x^{1+\theta(t+\varepsilon_x)/2}(\log y)^{\nu_t} + x^{\theta(t+\varepsilon_x)}(\log y)^{2t}
\end{align*}
Le choix $\theta = \frac{2}{t+\varepsilon_x}(1 - \frac{1}{\log\log x})$ fournit
\begin{equation}\label{eq:cond3form}
    \sum_{d\leq y^{u'\theta}}\left|A_d-\frac{\gamma^*(d)}{d}x^2\right| \ll \frac{x^2}{\log y}.
\end{equation}
\end{proof}

\begin{proof}[Démonstration du Théorème~\ref{theo:Formebinaire}]

D'après ce qui précède, le Théorème~\ref{theo:principal} s'applique avec le choix de $\theta$ ci-dessus et $X = x^2$, ce qui fournit 
\begin{align*}
    \Psi_F(x,y) &\geq \varrho_{\theta,1}(u')x^2 + o(x^2)
\end{align*}
lorsque $x\to \infty$, valable dans le domaine défini par $(1-\theta)u' \leq 1$, qui se traduit par $x\geq y^{t-2}$.\\
Dans le cas où $t=2$, nous avons $\theta = 1-\varepsilon_x'$ avec $\varepsilon_x'\tendvers{x}{+\infty} 0$, donc d'après le point~\eqref{eq:theta1} du Corollaire~\ref{cor:rhotheta}, 
\begin{align*}
    \Psi_F(x,y) &\geq \{\varrho(u')+o(1)\}x^2 + o(x^2)\\
    &= \varrho(2u+o(1))x^2 + o(x^2)\\
    &= \varrho(2u)x^2 + o(x^2).
\end{align*}
Dans le cas où $t=3$, nous avons $\theta = 2/3-\varepsilon_x''$ avec $\varepsilon_x''\tendvers{x}{+\infty} 0$, donc d'après le point~\eqref{eq:theta23} du Corollaire~\ref{cor:rhotheta},
\begin{align*}
    \Psi_F(x,y) &\geq f(u)x^2 + o(x^2)
\end{align*}
où
$$f(u)=\begin{cases}
                    1-u,&\quad\text{si } \frac{1}{3}<u\leq \frac{1}{2}\\
                    1-u+\log\left(\frac{1}{3u}+\frac{1}{3}\right),&\quad\text{si } \frac{1}{2}<u\leq \frac{2}{3}\\
                    1-u+\log\left(\frac{1}{3u}+\frac{1}{3}\right) + u\log\frac{3u}{2},&\quad\text{si } \frac{2}{3}<u\leq 1.
                \end{cases}$$
Dans le cas où $t\geq 4$, nous avons $\theta <1/2$, donc d'après le point~\eqref{eq:theta12} du Corollaire~\ref{cor:rhotheta}, 
\begin{align*}
    \Psi_F(x,y) &\geq (1-(t-2)u)x^2 + o(x^2).
\end{align*}

\end{proof}

\begin{remark}
    Les cas $t\leq 3$ sont traités ici mais ils sont dépassés par les estimations beaucoup plus précises de Lachand (voir remarque~\ref{rem:Lachand}).  
\end{remark}

\section{Application aux entiers friables voisins}\label{sec:voisins}

Dans cette partie, nous démontrons le Théorème~\ref{theo:HaretPas} et le Corollaire~\ref{cor:Wan}.\\
Soient $a\in\Z^*$ et $x,u,v\geq 1$. Nous étudions la quantité 
$$\Psi(x;u,v) := \{n\leq x:P^+(n)\leq x^{1/u}\mbox{ et } P^+(n+a)\leq x^{1/v}\}.$$
Il faut ainsi choisir $\mathcal{A} = \{n\leq x:P^+(n+a)\leq x^{1/v}\}$, si bien que 
$$\Psi(x;u,v) = \Psi_{\mathcal{A}}(y).$$
Pour simplifier les notations, notons $z = x^{1/v}$. Afin d'appliquer le Théorème~\ref{theo:principal}, il faut pouvoir estimer $$A_d = \#\{n\leq x+a:P^+(n)\leq z,\text{ et }n\equiv a\bmod d\} = \Psi(x,z;a,d) + O(1)$$ où nous avons posé 
$$\Psi(x,z;a,d):= \#\{n\leq x:P^+(n)\leq z,\text{ et }n\equiv a \bmod d\}.$$
Heuristiquement, nous nous attendons à ce que les entiers $z$-friables soient équirépartis dans les classes modulo $d$, soit pour $(a,d)=1$
\begin{equation}\label{eq:approxPhi}
    \Psi(x,z;a,d) \overset{?}{\sim} \frac{\Psi_d(x,z)}{\varphi(d)}.
\end{equation}
Nous savons que c'est la cas lorsque le module $d$ considéré est fixé et plus généralement pour des modules vérifiant $d\leq z^{4\sqrt{\ee}-\delta}$ pour $\delta>0$ fixé et~$\log x/\log d \to\infty$ \cite{Sou08,Har12b}. Nous disposons également de théorèmes de majoration de l'erreur en moyenne.
\begin{defn}
    Pour $\theta\in\oo{0,1}$ et $\Lambda\subset \fo{1,+\infty}^2$, nous disons que les entiers friables ont un exposant de répartition $\theta$ dans le domaine $\Lambda$ s'il existe $\delta>0$ tel que la majoration suivante est valable pour tous $(x,z)\in \Lambda$, $a\in\Z^*$:
$$\ssum{q\leq x^\theta\\(a,q) = 1}\left| \Psi(x,z;a,q)-\frac{\Psi_q(x,z)}{\varphi(q)} \right| \ll_{\theta,a} \frac{\Psi(x,z)}{(\log x)^\delta}.$$
\end{defn}
\begin{remark}
    Dans la littérature, il est plus fréquent de rencontrer un majorant plus fort, de la forme~$\Psi(x,z)/(\log x)^A$ pour tout $A>0$, mais ce n'est pas utile pour notre application.
\end{remark}
La recherche d'un meilleur exposant $\theta$ possible a une longue histoire en commençant par les travaux de Wolke en 1973 \cite{Wol73,FT91,FT96,Gra93a,Gra93b,Har12a,Dra15} et enfin tout récemment, Pascadi~\cite{Pas24} obtient l'exposant $\theta = 5/8-\varepsilon$, à condition que $(\log x)^C \leq z\leq x^{1/C}$ pour un certain $C=C(A,\varepsilon)>0$ (voir Théorème~\ref{theo:Pascadi}). Dans un article à paraître, Pascadi montre que l'exposant $\theta=5/8-\varepsilon$ est valable, toujours sous la conddition $(\log x)^C \leq z\leq x^{1/C}$ pour un certain $C=C(A,\varepsilon)>0$.  Il est à noter que nous sommes encore loin de la conjecture d'Elliot-Halberstam pour les entiers friables (Conjecture~\ref{conj:EH} ci-dessous).  \\

Pour pouvoir appliquer le Théorème~\ref{theo:principal}, nous choisissons la fonction $\gamma$ comme suit
\begin{equation}\label{eq:gammavoisins}
    \gamma(d) := \frac{d}{\varphi(d)}g_d(x,z)
\end{equation}
où 
$$g_d(x,z) := \prod_{p|d}\left(1-\frac{1}{p^{\alpha(x,z)}}\right)$$
où $\alpha(x,z)$ est le point col (voir \cite{Ten15}, Tome III, Chapitre 5, Théorème 5.2), unique solution de l'équation 
$$\sum_{p\leq z}\frac{\log p}{p^\alpha - 1} = \log x.$$
En particulier, $\alpha(x,z)\in\oo{0,1}$, ce qui entraîne $g_d(x,z) < \varphi(d)/d$.
\begin{prop}\label{prop:gammavoisins}
    Soit $\varepsilon>0$. La fonction $\gamma$ est multiplicative et vérifie les conditions~\eqref{cond1} avec~$\kappa=1$ et~\eqref{cond2}. De plus, si les entiers friables ont un exposant de répartition $\theta\in\of{0,1}$ dans un domaine $$\Lambda\subset \{x\geq 1:(\log x)^{1+\varepsilon}<z<x\},$$ alors $\gamma$ vérifie la condition~\eqref{cond4} avec le choix~$X=\Psi(x,z)$.
\end{prop}
\begin{proof}
    La multiplicativité de $\gamma$ découle de celles des fonctions $\varphi$ et $d\mapsto g_d(x,z)$. Les conditions~\eqref{cond1} et $\eqref{cond2}$ se déduisent directement du fait que pour tout premier $p$ et $a\geq 1$,
    $$\gamma(p^a) = \gamma(p) = 1 + O\left(\frac{1}{p^{\alpha(x,z)}}\right)\quad \text{et} \quad \alpha(x,z) \gg_\varepsilon 1.$$
    Pour la minoration de $\alpha(x,z)$, le lecteur pourra se référer par exemple à~\cite{Ten15}, Tome III, Chapitre 5, équation (5.74).
    Il reste maintenant à vérifier la condition de majoration~\eqref{cond4}. En choisissant la fonction $X := \Psi(x,z)$, nous avons pour $d\geq 1$,
\begin{align*}
    r_d = A_d - \Psi(x,z)\frac{g_d(x,z)}{\varphi(d)}.
\end{align*}
L'inégalité triangulaire et le fait que $A_d = \Psi(x,z;a,d) + O(1)$ entraînent
\begin{equation}\label{eq:Fxzd}
    |r_d| \leq \left|\Psi(x,z;a,d) - \frac{\Psi_d(x,z)}{\varphi(d)}\right| + \frac{1}{\varphi(d)}\left|\Psi_d(x,z) -\Psi(x,z)g_d(x,z)\right| + O(1).
\end{equation}
Notons 
$$F(x,z;d) := \frac{1}{\varphi(d)}\left|\Psi_d(x,z) -\Psi(x,z)g_d(x,z)\right|.$$
Une étude détaillée du comportement de la quantité $\Psi_d(x,z)$ a été mené par La Bretèche et Tenenbaum dans~\cite{dlBT05}. En particulier, si $k:=\Omega(d)$ est fixé, alors $\gamma(d) \ll_k 1$ et le Théorème~$2.1$ de~\cite{dlBT05} associé au Corollaire $2.2$ \emph{ibid.} permettent dans ce cas d'écrire
$$F(x,z;d) \ll_k \frac{g_d(x,z)}{\overline{\alpha}v\varphi(d)}\frac{\Psi(x,z)}{\log z}\ll_k\frac{1}{\overline{\alpha}d}\frac{\Psi(x,z)}{\log x}$$
où la notation $\overline{\alpha}$ est empruntée à \cite{dlBT05}:
$$\overline{\alpha} := \frac{\log(1+\frac{z}{\log x})}{\log z}> \frac{\varepsilon}{1+\varepsilon} + O\left(\frac{1}{\log z}\right)\gg_\varepsilon 1,$$
la minoration venant du fait que $z>(\log x)^{1+\varepsilon}$. Ainsi, à $\varepsilon>0$ fixé,
\begin{align*}
    \ssum{d\in S(y^{u\theta},y)\\(a,d)=1\\ \Omega(d) = k}F(x,z;d) & \ll_k \frac{\Psi(x,z)}{\log x}\ssum{d\in S(y^{u\theta},y)\\(a,d)=1\\ \Omega(d) = k}\frac{1}{d}\\
    &\ll_k \frac{\Psi(x,z)}{\log x}\ssum{d\leq y^{u\theta}\\\Omega(d)=k}\frac{1}{d}\\
    & \ll_k\frac{\Psi(x,z)}{\log x}(\log\log y)^k.
\end{align*}
Il s'ensuit que pour tout $0<\delta<1$,
\begin{equation}\label{eq:majFxzd}
    \ssum{d\in S(y^{u\theta},y)\\(a,d)=1\\ \Omega(d) = k}F(x,z;d) \ll_k \frac{\Psi(x,z)}{(\log y)^\delta}.
\end{equation}
Finalement, de~\eqref{eq:Fxzd} et~\eqref{eq:majFxzd} il découle
\begin{align*}
    \ssum{d\in S(y^{u\theta},y)\\(a,d)=1\\ \Omega(d) = k}|r_d| &= \ssum{d\in S(y^{u\theta},y)\\(a,d)=1\\ \Omega(d) = k}\left|\Psi(x,z;a,d)-\frac{\Psi_d(x,z)}{\varphi(d)}\right| + O_k\left(\frac{\Psi(x,z)}{(\log y)^\delta}\right)\\
    &\leq \ssum{d\leq y^{u\theta}\\(a,d)=1}\left|\Psi(x,z;a,d)-\frac{\Psi_d(x,z)}{\varphi(d)}\right| + O_k\left(\frac{\Psi(x,z)}{(\log y)^\delta}\right)\\
    &\ll_{k,a} \frac{\Psi(x,z)}{(\log y)^\delta}
\end{align*}
grâce à l'hypothèse de l'énoncé selon laquelle les entiers friables ont un exposant de répartition~$\theta$ dans le domaine $\Lambda$. En particulier, la fonction $\gamma$ vérifie la condition~\eqref{cond4}.

\end{proof}

\subsection{Cas inconditionnel}\label{sec:voisinincond}

Nous démontrons ici le Théorème~\ref{theo:HaretPas}. Pour cela, nous avons besoin du résultat suivant dû à Harper:

\begin{theo}[\cite{Har12a}, Theorem 1]\label{theo:Harper}
Il existe $K,c>0$ tels que pour $x$ et $z$ suffisamment grands, vérifiant $(\log x)^K<z<x$ et $1\leq Q\leq \sqrt{\Psi(x,z)}$, pour tout $A>0$,
$$\sum_{q\leq Q}\underset{(a,q) = 1}{\max}\left| \Psi(x,z;a,q)-\tfrac{\Psi_q(x,z)}{\varphi(q)} \right| \ll_A \Psi(x,z)\left(\tfrac{ \ee^{\tfrac{-cv}{\log^2(v+1)}}}{\log^A x}+z^{-c}\right) + \sqrt{\Psi(x,z)}Q\log^{7/2}(x)$$
où nous rappelons que $v$ est défini par la relation $z:=x^{1/v}$.
\end{theo}

Dans un premier temps, nous énonçons un corollaire de ce théorème qui est plus maniable et plus adapté à notre application.

\begin{cor}\label{cor:Harper}
Pour tout $\varepsilon\in\oo{0,1/2}$, il existe $x_\varepsilon\geq 1$ et $K_\varepsilon>0$ tels que les entiers friables ont un exposant de répartition de $1/2-\varepsilon$ dans le domaine $$\Lambda_\varepsilon := \{x>x_\varepsilon,\,(\log x)^{K_\varepsilon}<z<x\}.$$
\end{cor}

\begin{proof}
Il suffit de montrer que pour $\varepsilon\in\oo{0,1/2}$, il existe $K_\varepsilon>0$ tel que pour tous~$(\log x)^{K_\varepsilon}<z<x$ suffisamment grands, 

$$\ssum{q\leq x^{1/2-\varepsilon}}\underset{(a,q)=1}{\max}\left| \Psi(x,z;a,q)-\frac{\Psi_q(x,z)}{\varphi(q)} \right| \ll_{\varepsilon} \frac{\Psi(x,z)}{\log x}. $$
    Soit $\varepsilon\in\oo{0,1/2}$. Prenons $K_\varepsilon > \max(1,\frac{1}{c},K,\frac{1}{\varepsilon})$ où $K>0$ et $c>0$ sont donnés par le Théorème~\ref{theo:Harper}. Pour $x\geq 3$, posons $Q:= x^{1/2-\varepsilon}$. Pour $z$ vérifiant $(\log x)^{K_\varepsilon}<z<x$,  
    \begin{equation*}
        \Psi(x,z) \geq \Psi(x,(\log x)^{K_\varepsilon})
    \end{equation*}
    par croissance de $\Psi(x,\cdot)$. D'après (1.14) de \cite{HT93}, lorsque $x\to \infty$
    \begin{equation}\label{eq:Psilogx}
    \Psi(x,z) > x^{1-\frac{1}{K_\varepsilon}+o(1)}>Q^2 .
    \end{equation}
    La dernière inégalité vient du fait que $K_\varepsilon>\frac{1}{\varepsilon}$.
    Ainsi, $Q\leq\sqrt{\Psi(x,z)}$ pour $x$ suffisamment grand, disons $x>x_\varepsilon$. Nous pouvons alors appliquer le Théorème~\ref{theo:Harper}. Or, l'expression en $v$ dans l'exponentielle est négative et la condition $z > (\log x)^{K_\varepsilon}$ implique que $z^{-c}$ est un $O(1/\log x)$ car $K_\varepsilon c>1$, donc
    \begin{align}\label{eq:HaperIntermediaire}
        \sum_{q\leq x^{1/2-\varepsilon}}\underset{(a,q) = 1}{\max}\left| \Psi(x,z;a,q)-\frac{\Psi_q(x,z)}{\varphi(q)} \right| &\ll_\varepsilon \frac{\Psi(x,z)}{\log x} + \sqrt{\Psi(x,z)}x^{1/2-\varepsilon}\log^{7/2}x.
    \end{align}
    Maintenant, puisque $\Psi(x,\cdot)$ est croissante et que $K_\varepsilon>1/\varepsilon$, l'inégalité~\eqref{eq:Psilogx} implique que pour~$x$ suffisamment grand
    $$\sqrt{\Psi(x,z)}x^{1/2-\varepsilon}\log^{7/2}x \ll_\varepsilon \frac{\Psi(x,z)}{\log x}$$
    Le résultat se déduit en réinjectant cette inégalité dans~\eqref{eq:HaperIntermediaire}.
    
\end{proof}

Un deuxième résultat récemment démontré par A. Pascadi (\cite{Pas24}, Theorem 1.1) permet de dépasser la ``barrière $1/2$'' du théorème de Harper au prix de se restreindre aux cas où $v=\frac{\log x}{\log z}$ est grand.

\begin{theo}[\cite{Pas24}, Theorem 1.1]\label{theo:Pascadi}
    Pour tous $A>0$ et $\varepsilon\in\oo{0,5/8}$, il existe une constante $C=C(A,\varepsilon)$ telle que pour dans le domaine défini par~$x>~2$ et $(\log x)^C\leq z\leq x^{1/C}$, 
    $$\ssum{q<x^{5/8-\varepsilon}\\(q,a)=1}\left|\Psi(x,z;a,q)-\frac{\Psi_q(x,z)}{\varphi(q)}\right|\ll_{a,A,\varepsilon}\frac{\Psi(x,z)}{(\log x)^A}.$$
\end{theo}
Un corollaire direct de ce théorème nous est utile:
\begin{cor}\label{cor:Pascadi}
    Pour tout $\varepsilon\in\oo{0,5/8}$, il existe une constante $C_\varepsilon$ telle que les entiers friables ont un exposant de répartition $5/8-\varepsilon$ dans le domaine $\Lambda_\varepsilon := \{x\geq 1: (\log x)^{C_\varepsilon} < z < x^{1/C_\varepsilon}\}$.
\end{cor}
 
\begin{proof}[Démonstration du Théorème~\ref{theo:HaretPas}]
Soit $\varepsilon\in\oo{0,1/2}$. Montrons tout d'abord la minoration~\eqref{eq:cas1HaretPas} du Théorème~\ref{theo:HaretPas}. La Proposition~\ref{prop:gammavoisins} et le Corollaire~\ref{cor:Harper} impliquent que la fonction $\gamma$ définie en~\eqref{eq:gammavoisins} vérifie les conditions~\eqref{cond1},~\eqref{cond2} et~\eqref{cond4} pour le choix $\theta = 1/2-\varepsilon$ et dans le domaine $$\Lambda_\varepsilon := \{x>x_\varepsilon:(\log x)^{K_\varepsilon}<z<x\}.$$ Nous sommes donc en mesure d'appliquer le Théorème~\ref{theo:principal}, ce qui fournit pour $1\leq v =o\left(\frac{\log x}{\log\log x}\right)$,
\begin{align*}
    \Psi(x;u,v)&\geq \varrho_{1/2-\varepsilon}(u)\Psi(x,x^{1/v})\\
    &=\left(1-\frac{u}{2}-\varepsilon u+o(1)\right)\Psi(x,x^{1/v})
\end{align*}
lorsque $x\to +\infty$ (l'explicitation de $\varrho_{1/2-\varepsilon}(u)$ étant donnée par~\eqref{eq:theta12}). Il suffit ensuite de faire dépendre $\varepsilon$ de $x$ de manière à ce qu'il tende vers $0$ suffisamment lentement pour que le corollaire~\ref{cor:Harper} reste vrai (un choix convenable est $\varepsilon = 1/\log\log x$).\\

Le même raisonnement s'applique pour montrer la minoration~\eqref{eq:cas3HaretPas} en combinant la Proposition~\ref{prop:gammavoisins} et le Corollaire~\ref{cor:Pascadi}, à la différence près que le domaine de validité pour $v$ devient~$C_\varepsilon\leq v=o\left(\frac{\log x}{\log\log x}\right)$, la constante $C_\varepsilon$ étant celle du théorème de Pascadi qui tend vers l'infini lorsque $\varepsilon$ tend vers $0$.\\

\end{proof}

\subsection{Cas conditionnel}\label{sec:voisincond}

Passons maintenant à la démonstration du Corollaire~\ref{cor:Wan}.\\

Les théorèmes de Harper et Pascadi ci-dessus sont des tentatives pour obtenir le meilleur exposant de répartition dans le domaine le plus large possible. En fait, il est conjecturé que pour tout~ $\varepsilon>0$ fixé, les entiers friables ont un exposant de répartition $\theta=1-\varepsilon$. C'est ce qui fait l'objet de la conjecture d'Elliot-Halberstam pour les entiers friables:
\begin{conj}[$EH_{friable}$]\label{conj:EH}
Soit $\varepsilon\in\oo{0,1}$. Si le rapport $v=\log x/\log z$ est fixé, alors
$$\sum_{q\leq x^{1-\varepsilon}}\underset{x'\leq x}{\max}\;\underset{(a,q)=1}{\max}\left|\Psi(x',z;a,q)-\frac{\Psi_q(x',z)}{\varphi(q)}\right|\ll_{A,\varepsilon}\frac{x}{\log^A x}.$$
\end{conj}
Il est aisé de voir que sous la conjecture ci-dessus, les entiers friables ont un exposant de répartition $1-\varepsilon$ dans le domaine $\{(x,x^{1/v}); x\geq 1\}$ pour tout $\varepsilon\in\oo{0,1}$. La Proposition~\ref{prop:gammavoisins}, le Théorème~\ref{theo:principal} et le Corollaire~\ref{cor:rhotheta} nous donnent alors lorsque $x\to\infty$
\begin{align*}
    \Psi(x;u,v) &\geq (\varrho_{1-\varepsilon,1}(u)+o(1))\Psi(x,x^{1/v})\\
    &\geq (\varrho(u) + O(\varepsilon))\Psi(x,x^{1/v})\\
    &\geq (\varrho(u) + O(\varepsilon))(\varrho(v)x + o(x))
\end{align*}
et il suffit de faire dépendre $\varepsilon$ de $x$ de façon à ce qu'il tende suffisamment lentement vers~$0$ lorsque~$x$ tend vers l'infini. Nous retrouvons bien la minoration dans le résultat, plus général, de Wang \cite{Wan21} qui fournit un équivalent asymptotique. \\

Il est à noter que la conjecture ci-dessus est fausse avec $\varepsilon=0$. Cependant, il semble possible d'autoriser $\varepsilon$ à tendre vers $0$ mais pas trop vite. Par exemple, un résultat de Fiorilli et La Bretèche \cite{FdlB20} implique que nous ne pouvons prendre $\varepsilon$ de l'ordre de $\log\log x/\log x$.

\begin{remark}
    Dans un travail indépendant de Bharadwaj et Rodgers~\cite{BR24}, les auteurs étudient une généralisation du problème étudié ici dans le cas particulier où l'exposant de répartition $\theta$ peut être choisi arbitrairement proche de $1$. En particulier, il est possible d'obtenir une autre preuve du théorème de Wang \cite{Wan21} avec leur méthode. Ils appliquent cette méthode à l'étude des valeurs friables de $p-1$. Nous n'avons pas exploré cette question ici car les méthodes de cribles pour les premiers donnent de meilleurs résultats, voir \cite{Lic22} pour des références. Dans le cas des valeurs de polynômes irréductibles, les auteurs de \cite{BR24} remarquent que l'hypothèse~\eqref{cond3} avec $\theta$ proche de $1$ est irréaliste, cependant il est plausible que la condition plus faible~\eqref{cond4} soit vraie pour des valeurs de~$\theta$ plus grandes que $1/g$.

\end{remark}

\subsection*{Remerciements}

Je tiens à exprimer ici mes sincères remerciements à mon directeur de thèse, Sary Drappeau, pour son aide bienveillante, ses précieux conseils et sa présence attentive lors de la réalisation de ce travail. Je remercie également Régis de la Bretèche, Cécile Dartyge et le relecteur anonyme pour les remarques et les conseils avisés. Ma reconnaissance se tourne également vers Alexandru Pascadi pour son amitié et pour m'avoir fait part de son travail lors de son séjour à Marseille. Une partie de ce travail a été réalisée lors de mon stage de M2 et de ma thèse, financés par l'Institut Fourier (UMR~5582), l'Institut de Mathématiques de Marseille (UMR 7373) et l'ED 184 de l'université d'Aix-Marseille.



\end{document}